\documentclass[preprint,12pt]{elsarticle}
\usepackage{amsmath}
\usepackage{amsfonts}
\usepackage{algorithm}
\usepackage{mathrsfs}
\usepackage{algorithm}
\usepackage{algorithmic}
\usepackage{amsthm}
\usepackage[hidelinks]{hyperref}
\usepackage{tikz}
\usepackage{graphicx}
\usepackage{amssymb}
\usepackage{adjustbox}
\usepackage{booktabs}
\usepackage{array}
\usepackage{multirow}
\usepackage{caption}
\usepackage{subcaption}
\usepackage[a4paper,tmargin=1.2in, bmargin=1in, lmargin=1.1in, rmargin=1in, headheight=13.6pt]{geometry}
\journal{Journal of Mathematics}
\theoremstyle{plain}
\newtheorem{theorem}{\bfseries Theorem}[section]

\newtheorem{lemma}{\bfseries Lemma}[section]

\theoremstyle{remark}
\newtheorem{remark}{\bfseries Remark}[section]
\newtheorem{example}{\bfseries Example}[section]
\newtheorem{problem}{\bfseries Problem}[section]

\date{}

\begin{document}
\begin{frontmatter}
\title{Nearest structured matrix having an eigenvalue with prescribed lower bounds of algebraic and geometric multiplicities
}
\author[inst1]{H. Lalhriatpuia}
\affiliation[inst1]{
    organization={Department of Mathematics and Computer Science},
    addressline={Mizoram University},
    state={Mizoram},
    country={India}
}

\author[inst1]{Tanay Saha\corref{cor1}}
\cortext[cor1]{Corresponding author}

\author[inst2]{Punit Sharma}
\affiliation[inst2]{
    organization={Indian Institute of Technology Delhi},
    addressline={New Delhi},
    country={India}
}
\begin{abstract}
We study the problem of perturbing a given matrix $X$ belonging to a subspace $\mathcal{S}$ of linearly structured matrices in $\mathbb{R}^{n \times n}$ to its nearest counterpart $Y = X + \Delta$, where $\Delta \in \mathcal{S}$ and $Y$ possesses an eigenvalue with prescribed lower bounds on its algebraic multiplicity and geometric multiplicity.
The proposed framework takes care of both cases where the target eigenvalue is known and where it is treated as an unknown decision variable to be determined jointly with the structured perturbation. 
We establish necessary and sufficient conditions for the existence of such matrices by expressing the feasibility constraints through Jordan-chain relations in structural coordinates. This allows us to formulate the problem as a nonlinear constrained nested optimization that minimizes the Frobenius norm of the perturbation. 
To solve this, we implemented a two-level strategy, utilizing MATLAB’s \texttt{fmincon} for continuous inner optimization via a multi-start Sequential Quadratic Programming approach, while the outer level solves a discrete optimization problem. 
%
Comprehensive numerical experiments on various classes of structured matrices, including comparisons with existing methods, demonstrate the effectiveness and accuracy of the proposed approach. 
\end{abstract}

\end{frontmatter}
\section{Introduction}

Let $\mathcal S \subseteq \mathbb F^{n \times n}$, where $\mathbb F \in \{\mathbb R,\mathbb C\}$ denote a class of matrices that satisfies certain structure and let $\mathbb P$ be some matrix property. The nearness problem for matrices consists of finding, for a given matrix $X$, the nearest matrix (with respect to some prescribed norm) from the structured class $\mathcal S$ satisfying $\mathbb P$. If the distance between the given matrix $X$ and the nearest matrix from $\mathcal S$ is large, then the property $\mathbb P$ of $X$ is robust, and if the distance is small, then the original matrix is more likely to be ill-conditioned or more sensitive to perturbations, and some remedial actions need to be taken. The matrix nearness problems have been studied extensively in the literature, see~\cite{GilS2024,higham1988matrix} and references therein.


The solution to such nearness problems is useful in determining stability, resonance, and long-term behavior of dynamical systems. In real-world settings, the matrices are not arbitrary but possess some structures. 
When such structures are present, they are often  preserved during perturbations or approximations, both for interpretability and for numerical stability \cite{gray2006toeplitz,horn2012matrix,trefethen2022numerical}.

In this paper, for a given matrix $X \in \mathcal{S} \subseteq \mathbb R^{n \times n}$, we study the problem of finding the smallest perturbation $\Delta \in \mathcal S$ such that $X+\Delta$ has an eigenvalue with prescribed lower bounds of algebraic and geometric multiplicities. More precisely, we consider the following problems:

Let $X \in \mathcal{S} $, $\lambda \in \mathbb R$, where $\mathcal S \subseteq \mathbb R^{n \times n}$ be a subspace of a particular class of linearly structured matrices. For given integers $m$ and $p$ satisfying $1 \le p \le m \le n$, let us define
\begin{equation}\label{V Set}
\mathcal{V} :=
\left\{
\Delta  \in \mathcal{S} : GM(\mu; X+\Delta) \ge p,\;
AM(\mu; X+\Delta) \ge m, \ \mu \in \mathbb{R} 
\right\}
\end{equation}
and
\begin{equation} \label{U Set}
\mathcal{V}_{\lambda}
:=
\left\{
\Delta \in \mathcal{S} :
GM(\lambda; X+\Delta) \ge p,\;
AM(\lambda; X+\Delta) \ge m
\right\},
\end{equation}
where $GM(\rho; X)$ and $AM(\rho; X)$, respectively, stand for the \emph{geometric multiplicity (GM)} and the \emph{algebraic multiplicity (AM)} of $\rho$ as an eigenvalue of the matrix $X$.

\begin{problem} [Finding the nearest structured matrix having an eigenvalue with prescribed lower-bounded multiplicities]
For a given matrix $X \in  \mathcal S \subseteq \mathbb R^{n \times n}$, find an optimal perturbation $ \widehat \Delta \in \mathcal{V}$ with minimal Frobenius norm. Specifically, we aim to compute
\begin{equation}\label{eq:prob1}
    \widehat \Delta \in \arg \min_{\Delta \in \mathcal{V}} {\|\Delta\|}_F.\tag{$\mathcal P$}
\end{equation}
where ${\|\cdot\|}_F$ stands for the Frobenius norm. 
\end{problem}

\begin{problem}[Finding the nearest structured matrix having a given eigenvalue with prescribed lower-bounded multiplicities]
For a given matrix $X \in \mathcal{S} \subseteq \mathbb R^{n\times n}$, and  a target eigenvalue $\lambda \in \mathbb R$, find an optimal perturbation $ \widehat \Delta_\lambda \in \mathcal{V}_{\lambda}$ with minimal Frobenius norm. Specifically, we aim to compute
\begin{equation}\label{eq:prob2}
    \widehat \Delta_\lambda \in \arg \min_{\Delta \in \mathcal V_{\lambda} }{\|\Delta\|}_F.\tag{$\mathcal P_\lambda$}
\end{equation}
\end{problem}



The problem of computing the nearest matrix with prescribed eigenvalues has a rich history. Malyshev \cite{malyshev1999formula} gave a singular-value formula for the two-norm distance from a matrix to the set of matrices with multiple eigenvalues, providing a rigorous form of the Wilkinson distance. 
Higham, in his Ph.D. thesis \cite{higham1985nearness} and subsequent works \cite{higham1988computing,higham1988matrix,higham2002computing}, investigated several nearness problems, providing a key connection to the problems considered in this work. Extensions have considered fixing two eigenvalues by minimal perturbations \cite{lippert2005fixing}, as well as the computational aspects of determining the nearest matrix with two prescribed eigenvalues \cite{nazari2010computational}. The problem has also been investigated for eigenvalues with bounded multiplicities \cite{armentia2020nearest}, and Mengi~\cite{mengi2011locating}  studied the nearest matrix problem with prescribed algebraic multiplicity. Meanwhile, Saha~\cite{Saha24082026} considered the nearest structured polynomial matrix problem with an eigenvalue having prescribed geometric multiplicity.
More recent contributions include computational schemes for the nearest matrix with prescribed partial eigenvalues and its applications \cite{gracia2005nearest,kokabifar2016nearest,lalhriatpuia2025symmetric}. Together, all these works highlight both theoretical characterizations and computational strategies for matrix nearness problems under diverse spectral and structural constraints.

Some studies have also focused on structure-preserving versions of the distance problems on matrices. 
Such formulations are important in physical applications, where the perturbation must maintain symmetry, Toeplitz, or Hamiltonian structures \cite{borsdorf2012structured,mehrmann2001structure,nunez2019structured}. 
These works showed that structural constraints can significantly change both the geometry and difficulty of the problem, making it essential to design algorithms that respect the given linear structures of the matrices.


However, most studies assume that the eigenvalues are known or fixed. 
In many cases, the desired multiplicity structure is specified while the eigenvalue $\lambda$ is given \cite{armentia2020nearest,butta2015differential,kressner2014generalized,mengi2011locating}, and the structure of the updated matrix is not preserved. 
When an eigenvalue with a specified multiplicity structure needs to be determined while preserving the inherent matrix structure, the problem becomes significantly more challenging and, to the best of our knowledge, remains open. This situation often occurs in applications where the eigenvalue is not fixed but arises naturally from the structural constraints of the system. 

\subsection{Contribution and outline of the paper}
In Section~\ref{section 2}, we introduce the notation and definitions that will be used throughout the paper.

To address the problems~\eqref{eq:prob1} and~\eqref{eq:prob2}, in Section~~\ref{sec:theorem}, we provide a complete characterization for the existence of matrices in $\mathcal V$ and $\mathcal V_\lambda$. We establish necessary and sufficient conditions by expressing the feasibility constraints through Jordan-chain relations in structural coordinates.

This characterization is then used, in Section~\ref{Reformulation}, to reformulate the nearest matrix problems in~\eqref{eq:prob1} and~\eqref{eq:prob2} as a nonlinear constrained optimization problem that minimizes the Frobenius norm of the structured perturbation.
To solve the reformulated problem, we propose a two-level strategy consisting of a continuous inner optimization, solved using MATLAB's \texttt{fmincon} with a multi-start Sequential Quadratic Programming scheme, and a discrete outer optimization that systematically explores all admissible Jordan-chain configurations generated from integer partitions satisfying the prescribed multiplicity constraints. To the best of our knowledge, this is the first work to tackle the problem where the target eigenvalue with specified multiplicities is unknown and  determined jointly with the structured perturbation. 

The approach is validated through extensive numerical experiments, in Section~\ref{sec numerical}, on various classes of structured matrices and comparisons with existing methods, confirming both the feasibility and accuracy of the proposed framework.
We show that the proposed framework unifies and generalizes earlier formulations, naturally reducing to classical results when $\lambda$ is fixed or becoming a joint optimization when $\mu$ is treated as a variable. In several numerical examples, the optimal matrix satisfying the multiplicity constraints coming out to be a defective matrix. This shows that the solutions to problems~\eqref{eq:prob1} and~\eqref{eq:prob2} bound the structured distance to defectivity from above. 


\section{Preliminaries}\label{section 2}
In this section, we introduce the basic notation and terminology in the form that will be used throughout the paper. 

\subsection{Vectorization of a Matrix}

An important tool in our reformulation of problems~\eqref{eq:prob1} and~\eqref{eq:prob2} is the \emph{vectorization operator}. 
For a matrix $A = (x_{ij}) \in \mathbb{R}^{n \times n}$, 
the vectorization $\mathrm{vec}(A) \in \mathbb{R}^{n^2}$ is obtained 
by stacking the columns of $A$ on top of their consecutive columns
\[
\mathrm{vec}(A) := 
\begin{bmatrix}
x_{11} & x_{21} & \cdots & x_{n1} &
x_{12} & x_{22} & \cdots & x_{n2} &
\cdots &
x_{1n} & x_{2n} & \cdots & x_{nn}
\end{bmatrix}^{\top}.
\]
For conformable matrices $A,B,C$ over $\mathbb{R}$, the following vectorization identity holds.
\begin{equation}\label{eq:veciden}
  \operatorname{vec}(A B C) = (C^\top \otimes A) \operatorname{vec}(B),
\end{equation}
where $\otimes$ is the Kronecker product.

\subsection{Space of Structured Matrices}

Suppose $A$ belongs to a subspace $\mathcal{S} \subseteq \mathbb R^{n\times n}$ of linearly structured matrices and let $l=\text{dim}(\mathcal S)$. If 
$E_1, E_2, \dots, E_l$ be the basis matrices defining the structure class $\mathcal S$, then any matrix $A \in \mathcal{S}$ can be expressed as
\begin{equation}\label{eq:basisrep}
    A = \sum_{j=1}^l \alpha_j E_j, \quad \alpha_j \in \mathbb{R}.
\end{equation}
The coefficients $\alpha_j$'s are called as the \emph{structured parameters} with respect to the basis $E_1, E_2, \dots, E_l$. By defining the \emph{transformation matrix} 
\begin{equation}\label{Dmatrix}
    D := \begin{bmatrix}
    \mathrm{vec}(E_1) & \mathrm{vec}(E_2) &  \cdots & \mathrm{vec}(E_l)
\end{bmatrix},
\end{equation}
and the \emph{parameter vector}
\begin{equation}\label{vec1matrix}
    \mathrm{vec}_1(A) := (\alpha_1, \alpha_2, \dots,\alpha_l)^\top,
\end{equation}
the relation in~\eqref{eq:basisrep} can be expressed as
\begin{equation}
    \mathrm{vec}(A) = D \, \mathrm{vec}_1(A).
\end{equation}

 \begin{example}[Symmetric Matrices]
 Let $\mathcal S$ be the subspace of real symmetric matrices of size $3 \times 3$ and let $ A\in \mathcal S$ with the structure parameters
 $\{a,b,c,d,e,f\}$ corresponding to the standard basis $\{E_{11},E_{12}+E_{21},E_{13}+E_{31},E_{22},E_{23}+E_{32},E_{33}\}$ of $\mathcal S$, where $E_{ij}$ has zero everywhere except a one at  $\{ij\}^{th}$ position. Then 
 \[
 \mathrm{vec}_1(A) =
 \begin{bmatrix}
 a & b & c & d & e & f
 \end{bmatrix}^\top, \quad
 \mathrm{vec}(A) =
 \begin{bmatrix}
 a & b & c & b & d & e & c & e & f
 \end{bmatrix}^\top,
 \]
and we have
 \[
 \mathrm{vec}(A) = D \, \mathrm{vec}_1(A),
 \]
 where $D$ is the transformation matrix given by
 \[
 D =
 \begin{bmatrix}
\mathrm{vec}(E_{11}) & \mathrm{vec}(E_{12}+E_{21})&\mathrm{vec}(E_{13}+E_{31})&\mathrm{vec}(E_{22})&\mathrm{vec}(E_{23}+E_{32})&\mathrm{vec}(E_{33})
 \end{bmatrix}.
 \] 
 \end{example}

\section{Characterization of matrices in $\mathcal{V}$ and $\mathcal{V}_\lambda$}\label{sec:theorem}

In this section, we present the main theoretical results of the paper and establish that finding a matrix in $\mathcal{V}$ or  $\mathcal{V}_\lambda$ is equivalent to the feasibility of a nonlinear system of vector equations. 

The following lemma gives a necessary and sufficient condition for a matrix $Y$ to have an eigenvalue $\mu$ with prescribed lower bounds of algebraic and geometric multiplicities. 

\begin{lemma}
\label{lemma:GM_AM_characterization_fixed}
Let $Y\in\mathbb{R}^{n\times n}$ and let integers $1\le p\le m\le n$ be given. Then $Y$ has an eigenvalue $\mu \in \mathbb{R}$ with 
\[
\mathrm{GM}(\mu;Y)\ge p \quad\text{and}\quad \mathrm{AM}(\mu;Y)\ge m
\]
if and only if there exist an integer $r$ with $m\ge r \ge p$,  positive integers $k_1\ge k_2\ge\cdots\ge k_r \geq 1$ such that
\[
\sum_{i=1}^r k_i = m,
\]
and vectors
\[
\{v_j^{(i)}\in\mathbb{R}^n:\; i=1,2,\dots,r;~ j=1,2,3,\dots,k_i\}
\]satisfying the following conditions:
\begin{enumerate}
\item[{\rm(a)}] $(Y-\mu I)v_1^{(i)}=0$ for $i=1,2,\dots,r$.
 \item[{\rm(b)}] The set of vectors
 $\{v_1^{(i)} : i = 1,2, \dots, r\}
 \ \text{is linearly independent in } \mathbb{R}^n$.
 \item[{\rm(c)}] $(Y-\mu I)v_j^{(i)}=v_{j-1}^{(i)}$ for each $i=1,2,\dots,r$ and $j=2,3,\dots,k_i$.
\end{enumerate}
\end{lemma}
\begin{proof}
($\Rightarrow$)
First suppose that $\mu$ is an eigenvalue of $Y$ with
$\mathrm{GM}(\mu;Y) \ge p$ and $\mathrm{AM}(\mu;Y) \ge m$.
Let $r_0 = \mathrm{GM}(\mu;Y)$.
In the Jordan canonical form of $Y$, there exist $r_0$ Jordan blocks corresponding to the eigenvalue $\mu$ with block sizes
\[
\ell_1, \ell_2, \dots, \ell_{r_0} \ge 1,
\qquad \text{and} \qquad
\sum_{i=1}^{r_0} \ell_i = \mathrm{AM}(\mu;Y).
\]
Since $\mathrm{AM}(\mu;Y) \ge m$, we may choose an integer
$r$ with $p \le r \le r_0$ and integers
$k_i$ satisfying $1 \le k_i \le \ell_i$ for $i = 1,2, \dots, r$ with 
\[
\sum_{i=1}^r k_i = m.
\]
By reordering these $r$ numbers if necessary, w.l.o.g., we can ensure that the sequence satisfies $k_1 \ge k_2 \ge \dots \ge k_r.$

For each selected $i^{th}$ block, truncate the corresponding Jordan chain of generalized eigenvectors after
$k_i$ vectors, yielding vectors
$\{ v_j^{(i)} \}_{j=1}^{k_i}$.
These vectors satisfy
\[
(Y - \mu I)v_1^{(i)} = 0,
\qquad
(Y - \mu I)v_j^{(i)} = v_{j-1}^{(i)}, \quad j = 2,3,\dots, k_i,
\]
and the vectors $\{ v_1^{(i)} \}_{i=1}^r$ are linearly independent.
This proves the `if' part.

\medskip

($\Leftarrow$)
Conversely, suppose that there exist integers $m \ge r \ge p$ and
$k_1\ge k_2 \ge \cdots \ge k_r \ge 1$ with $\sum_{i=1}^r k_i = m$, together with vectors
$\{ v_j^{(i)} \}$ satisfying (a)-(c). The conditions in (a) and (b) imply that 
$\{ v_1^{(i)} \}_{i=1}^r$ are $r$ linearly independent eigenvectors of $Y$
corresponding to $\mu$.
Thus, we have
\[
\mathrm{GM}(\mu;Y) \ge r \ge p.
\]
Further, the condition in (c) shows that each collection
$\{ v_j^{(i)} \}_{j=1}^{k_i}$ forms some part of $i^{th}$ Jordan chain of length $k_i$.
Therefore, the generalized eigenspace of $Y$ associated with $\mu$
contains at least $\sum_{i=1}^r k_i = m$ linearly independent generalized
eigenvectors, implying
\[
\mathrm{AM}(\mu;Y) \ge m.
\]
This completes the proof.
\end{proof}

\begin{remark}\label{remark independent} 
Since the eigenvectors are inherently nonzero and may be scaled arbitrarily, a normalization condition can be imposed to exclude the vectors very close to zero. In Lemma~\ref{lemma:GM_AM_characterization_fixed},
the independent eigenvectors $\{v_1^{(1)}, v_1^{(2)}, \dots, v_1^{(r)}\}$ associated with $\mu$ may be normalized so that 
\begin{equation}
\label{normalization vector}
    {\|v_1^{(i)}\|}_2 = 1 \ \ \text{for all} \ \ i = 1,2, \dots, r,
\end{equation}
where ${\|\cdot\|}_2$ stands for the 2-norm. 
Furthermore, when desirable, 
they can be orthogonalized to satisfy 
\begin{equation}
\label{orhthogonality}
    v_1^{(i)\top} v_1^{(j)} = 0 \ \ \text{for} \ \ i \neq j
\end{equation}
without affecting the validity of the lemma.
\end{remark}

\begin{theorem} \label{thm:structured_existence_clean}
 Let  
$\mathcal{S} \subseteq \mathbb{R}^{n \times n}$ be the subspace of some linearly structured matrices and $D$ be the transformation matrix of a basis of $\mathcal{S}$ as defined in~\eqref{Dmatrix}. Let $X \in \mathcal{S}$ and let integers $m,p$ satisfy $1 \le p \le m \le n$. Then
there exist $\mu \in \mathbb R$ and a structured perturbation $\Delta \in \mathcal{S}$ such that $Y = X+\Delta$ has an eigenvalue $\mu $ satisfying
\[
\mathrm{GM}(\mu;Y)\ge p, \qquad \mathrm{AM}(\mu;Y)\ge m
\]
if and only if there exist integers  $m \ge r \ge p$, $k_1 \ge k_2 \ge \cdots \ge k_r$, with $\displaystyle \sum_{i=1}^{r}k_i=m$ and  a set of vectors $\{v_j^{(i)} :\; i=1,2,\dots,r;~ j=1,2,\dots,k_i\}$ satisfying~\eqref{normalization vector} and~\eqref{orhthogonality}, and a parameter vector $z \in \mathbb{R}^l$, such that the structured feasibility system
\begin{equation}\label{eq:structured_system_clean}
\mathcal{J}z = \mathcal{Q}
\end{equation}
is satisfied, where $\mathcal{J} \in \mathbb{R}^{nm \times l}$ and $\mathcal{Q} \in \mathbb{R}^{nm}$ are defined by
\begin{equation}\label{PQ eqn}
   \mathcal{J}=  \begin{bmatrix}
(v_1^{(1)\top}\!\otimes I_n)D\\
(v_2^{(1)\top}\!\otimes I_n)D\\
\vdots\\
(v_{k_r}^{(r)\top}\!\otimes I_n)D
\end{bmatrix},
\qquad
\mathcal{Q}=
\begin{bmatrix}
\mu v_1^{(1)}-Xv_1^{(1)}\\[3pt]
\mu v_2^{(1)}+v_1^{(1)}-Xv_2^{(1)}\\
\vdots\\
\mu v_{k_r}^{(r)}+v_{k_r-1}^{(r)}-Xv_{k_r}^{(r)}
\end{bmatrix},
\end{equation}
where $I_n$ is the identity matrix of size $n \times n$ and $\otimes$ is the Kronecker product. 
\end{theorem}
\begin{proof}
First observe that by taking $A=I_n$, $B=\Delta\in\mathbb{R}^{n\times n}$ and $C=v\in\mathbb{R}^n$
in the vectorization identity~\eqref{eq:veciden}, we have 
\begin{equation}\label{eq:vec_key}
\operatorname{vec}(\Delta v) = (v^\top \otimes I_n) \operatorname{vec}(\Delta).
\end{equation}
We will use this repeatedly in the proof. 
\medskip

($\Rightarrow$) \textbf{(From Multiplicities to the Feasibility System):} 
Suppose that for a given structured matrix $X \in \mathcal{S}$, there exists a structured perturbation $\Delta \in \mathcal{S}$ such that the perturbed matrix $Y = X + \Delta$ has an eigenvalue $\mu$ with $\mathrm{GM}(\mu;Y) \ge p$ and $\mathrm{AM}(\mu;Y) \ge m$. 
%
From Lemma~\ref{lemma:GM_AM_characterization_fixed}, it implies that there exist $r~ (\ge p)$ integers $k_1 \ge k_2 \ge \dots \ge k_r$ summing to $m$ and a set of vectors $\{v_j^{(i)}\}$ such that the leading vectors $\{v_1^{(1)}, v_1^{(2)}, \dots, v_1^{(r)}\}$ are linearly independent, and the following relations hold for each $i = 1,2, \dots, r:$
\begin{equation}\label{eq:temp1}
\begin{aligned}
(Y - \mu I_n)v_1^{(i)} &= 0, \\
(Y - \mu I_n)v_j^{(i)} &= v_{j-1}^{(i)}, \quad \text{for } j = 2,3, \dots, k_i.
\end{aligned}
\end{equation}
In view of Remark~\ref{remark independent}, the set $\{v_1^{(1)}, v_1^{(2)}, \dots, v_1^{(r)}\}$ can be normalized to satisfy~\eqref{normalization vector} and~\eqref{orhthogonality}.  
Substituting $Y = X + \Delta$ in~\eqref{eq:temp1} and isolating the terms involving $\Delta$ on the left side, we obtain
\begin{equation}\label{eq:temp2}
\begin{aligned}
\Delta v_1^{(i)} &= \mu v_1^{(i)} - X v_1^{(i)}, \\
\Delta v_j^{(i)} &= \mu v_j^{(i)} + v_{j-1}^{(i)} - X v_j^{(i)}, \quad \text{for } j = 2,3, \dots, k_i.
\end{aligned}
\end{equation}
Next, we apply the vectorization operator to both sides in~\eqref{eq:temp2}, and use the identity from \eqref{eq:vec_key} and the parametrization $\operatorname{vec}(\Delta) = Dz$ to obtain 
\begin{equation}\label{eq:temp3}
\begin{aligned}
(v_1^{(i)\top} \otimes I_n) Dz &= \mu v_1^{(i)} - X v_1^{(i)}, \\
(v_j^{(i)\top} \otimes I_n) Dz &= \mu v_j^{(i)} + v_{j-1}^{(i)} - X v_j^{(i)}, \quad \text{for } j = 2,3, \dots, k_i.
\end{aligned}
\end{equation}
For each fixed pair $(i, j)$, both these equations represent a block of $n$ scalar equations. By vertically stacking these blocks for $j = 2,3, \dots, k_i$ and $i = 1,2, \dots, r$, we form the large block-matrix system. The left-hand side matrices in~\eqref{eq:temp3} stack to form exactly $\mathcal{J}$ and the right-hand side vectors in~\eqref{eq:temp3} stack to form $\mathcal{Q}$ as defined in \eqref{PQ eqn}. This yields the system $\mathcal{J}z = \mathcal{Q}$. Since the valid perturbation $\Delta$ exists by assumption, this system is necessarily consistent for the associated $z \in \mathbb{R}^l$.

\medskip

\textbf{($\Leftarrow$) (From the Feasibility System to Multiplicities):} 
Conversely, suppose that there exist an integer $m \ge r \ge p$, integers $\{k_i\}_{i=1}^r$ summing to $m$, and vectors $\{v_j^{(i)}\}$ (with $\{v_1^{(i)}\}$ satisfying~\eqref{normalization vector} and~\eqref{orhthogonality}) such that $\mathcal{J}z = \mathcal{Q}$ is consistent for some $z \in \mathbb{R}^l$.

Construct $\Delta \in \mathcal{S}$ such that $\operatorname{vec}(\Delta) = Dz$, and let $Y = X + \Delta$. Because $\mathcal{J}z = \mathcal{Q}$ is consistent, every block equation within the stacked system holds. Reversing the vectorization process for the $j^{th}$ block of the $i^{th}$ system of equations
\[
(v_j^{(i)\top} \otimes I_n) Dz = (v_j^{(i)\top} \otimes I_n) \operatorname{vec}(\Delta) = \operatorname{vec}(\Delta v_j^{(i)}) = \Delta v_j^{(i)}.
\]
Equating this back to the corresponding block in $\mathcal{Q}$ recovers
\begin{align*}
\Delta v_1^{(i)} &= \mu v_1^{(i)} - X v_1^{(i)}, \\
\Delta v_j^{(i)} &= \mu v_j^{(i)} + v_{j-1}^{(i)} - X v_j^{(i)}, \quad \text{for }  j = 2,3, \dots, k_i.
\end{align*}
Rearranging the terms to group $X + \Delta$ yields the Jordan chain conditions for $Y$
\[
(Y - \mu I_n)v_1^{(i)} = 0 \quad \text{and} \quad (Y - \mu I_n)v_j^{(i)} = v_{j-1}^{(i)}.
\]
Since the $r$ leading vectors $\{v_1^{(i)}\}$ are linearly independent due to~\eqref{normalization vector} and~\eqref{orhthogonality}, from Lemma~\ref{lemma:GM_AM_characterization_fixed} we have $\mathrm{GM}(\mu;Y) \ge p$  and $\mathrm{AM}(\mu;Y) \ge m$. This completes the proof.
\end{proof}

Note that, in view of Theorem~\ref{thm:structured_existence_clean}, there exists a $\Delta \in \mathcal V$ if and only if the system of nonlinear vector equations $\mathcal{J}z = \mathcal{Q}$ is feasible. Note that 
the block rows of $\mathcal{J}$ and $\mathcal{Q}$ in  Theorem~\ref{thm:structured_existence_clean} comprise vectors
satisfying~\eqref{normalization vector} and~\eqref{orhthogonality}, resulting in $m$ vector equations. When the eigenvalue $\mu$ is unknown corresponding to Problem~\eqref{eq:prob1}, we group the unknowns into a mixed-domain decision variable 
    $y_1 \in \mathbb{R}^{nm+l+1}$ as follows
    \begin{equation}
    y_1 = \begin{bmatrix} v_1^{(1)} & 
    v_2^{(1)}&
    \cdots &
    v_{k_r}^{(r)} &
    z & \mu \end{bmatrix}^\top.
\end{equation}
    Thus, the feasibility condition of \eqref{eq:structured_system_clean} can be expressed as the root-finding problem 
    \begin{equation}
    \label{f y_1}
        f(y_1) := \mathcal{J}z - \mathcal{Q} = 0.
    \end{equation} 
    Therefore, finding a root for $f(y_1) = 0$ is precisely equivalent to identifying an element within the admissible set $\mathcal{V}$ defined in \eqref{V Set}. 

Similarly, when the eigenvalue $\lambda$ is known in advance 
corresponding to Problem~\eqref{eq:prob2}, let us introduce the variable $y_2 \in \mathbb{R}^{nm + l}$ which excludes the eigenvalue and consists only of the vectors 
$\{v_j^{(i)}\}$ and the structural parameters $z$, as
\begin{equation}
    y_2 =
\begin{bmatrix}
v_1^{(1)} &v_2^{(1)}& \cdots & v_{k_r}^{(r)} & z
\end{bmatrix}^\top.
\end{equation}
Thus, the feasibility condition can be written as the following root-finding problem
\begin{equation}\label{f(y_2)}
    f_\lambda(y_2) := \mathcal{J} z - \mathcal{Q}_\lambda = 0,
\end{equation}
where $\mathcal{Q}_\lambda$ denotes the matrix $\mathcal{Q}$ from~\eqref{PQ eqn} with the variable $\mu$ replaced by the given parameter $\lambda$. Therefore, finding a root for $f(y_2) = 0$ is precisely equivalent to identifying an element within the admissible set $\mathcal{V}_\lambda$ defined in \eqref{U Set}. 

As a summary, we have the following result that characterizes the existence of matrices in $\mathcal V$ (resp. $\mathcal V_\lambda$) as a feasibility of root-finding problem $f(y_1)=0$ 
in~\eqref{f y_1} (resp. $f_\lambda(y_2)=0$ in~\eqref{f(y_2)}).

\begin{theorem}\label{corollary_structured_solutions}
Let $X,p,m,\mathcal S$ and $D$ be as defined in Theorem~\ref{thm:structured_existence_clean}
and let $\mathcal V$ (resp. $\mathcal V_\lambda$) be defined by~\eqref{V Set}(resp.~\eqref{U Set}). Then 
$ \mathcal V$ (resp. $\mathcal V_\lambda$) is nonempty if and only if 
the nonlinear system $f(y_1) = 0$ defined by~\eqref{f y_1} (resp. $f_\lambda(y_2) = 0$ defined by~\eqref{f(y_2)}) admits a solution in the form 
\[
y_1= \begin{bmatrix} v_1^{(1)\top} & v_2^{(1)\top} & \dots & v_{k_r}^{(r)\top} & z_1^\top & \mu \end{bmatrix}^\top ~(resp.~ y_2= \begin{bmatrix} v_1^{(1)\top} & v_2^{(1)\top} & \dots & v_{k_r}^{(r)\top} & z_2^\top  \end{bmatrix}^\top),
\]
where the vectors $\{v_1^{(i)}\}_{i=1}^r$ satisfy the normalization and orthogonality conditions in \eqref{normalization vector} and \eqref{orhthogonality}. 

Moreover, for any solution $y_1$ of $f(y_1)=0$ (resp. $y_2$ of $f_\lambda(y_2)=0$), one can construct a corresponding structured perturbation $\Delta_1 \in \mathcal{V}$ (resp. $\Delta_2 \in \mathcal{V}_\lambda$) as $\Delta_1 = \mathrm{mat}(Dz_1)$ (resp. $\Delta_2 = \mathrm{mat}(Dz_2)$), where $\mathrm{mat}(\cdot)$ denotes the inverse of the vectorization operator and $z_1$ (resp. $z_2$) contains the structural parameters from $y_1$ (resp. $y_2$).
\end{theorem}
%



\section{Reformulation of the nearest matrix problems~\eqref{eq:prob1} and~\eqref{eq:prob2} } \label{Reformulation}

In this section, we utilize the characterization of matrices 
 within the sets $\mathcal{V}$ and $\mathcal{V}_\lambda$  (Theorem~\ref{corollary_structured_solutions}) to find the minimum-norm perturbation from $\mathcal{V}$ and $\mathcal{V}_\lambda$ by reformulating the nearest matrix problems~\eqref{eq:prob1}  and~\eqref{eq:prob2} as constrained nonlinear optimization problems. 

In view of Theorem~\ref{corollary_structured_solutions}, let $y_g$ (for $g \in \{1,2\}$) represent the solutions of $f(y_g) = 0$  for the unknown~\eqref{V Set} and known~\eqref{U Set} eigenvalue cases, respectively. In both cases, our objective is to minimize the squared Frobenius norm of the perturbation
\[
\mathbf{R}(z_g) := \|\Delta_g\|_F^2 = \|\operatorname{vec}(\Delta_g)\|_2^2 = \|Dz_g\|_2^2.
\]

From Theorem~\ref{corollary_structured_solutions}, finding an admissible perturbation $\Delta_1 \in \mathcal{V}$ or $\Delta_2 \in \mathcal{V}_\lambda$ is equivalent to satisfying the feasibility system $f(y_g) = \mathcal{J}z_g - \mathcal{Q} = 0$, alongside the normalization and orthogonality constraints given in ~\eqref{normalization vector} and~\eqref{orhthogonality}. Recall, the block matrices $\mathcal{J}$ and $\mathcal{Q}$ are defined in \eqref{PQ eqn} and given as 
\[
\mathcal{J} =  \begin{bmatrix} (v_1^{(1)\top} \otimes I_n)D \\ (v_2^{(1)\top} \otimes I_n)D\\[2pt] \vdots \\[2pt] (v_{k_r}^{(r)\top} \otimes I_n)D \end{bmatrix}, ~
\mathcal{Q} = \begin{bmatrix} \mu v_1^{(1)}-Xv_1^{(1)}\\[3pt] \mu v_2^{(1)}+v_1^{(1)}-Xv_2^{(1)}\\ \vdots\\ \mu v_{k_r}^{(r)}+v_{k_r-1}^{(r)}-Xv_{k_r}^{(r)} \end{bmatrix}, ~ \mathcal{Q}_\lambda = \begin{bmatrix} \lambda v_1^{(1)} - X v_1^{(1)} \\[2pt] \lambda v_2^{(1)}+v_1^{(1)}-Xv_2^{(1)}\\ \vdots \\[2pt] (\lambda v_{k_r}^{(r)} + v_{k_r-1}^{(r)}) - X v_{k_r}^{(r)} \end{bmatrix}.
\]
Therefore, problems~\eqref{eq:prob1} and~\eqref{eq:prob2} can be equivalently reformulated as a two-level optimization problem as follows.

{\it The inner optimization}:~For a fixed configuration with $m \ge r \ge p$, and a particular partition of $m$, say ${\Gamma}_k=(k_1,k_2,\ldots,k_r)$ satisfying $k_1 \ge k_2 \ge \dots \ge k_r$ and $\sum_{i=1}^r k_i = m$, the inner optimization problem can be stated as follows:
\begin{equation}
\label{INNER}
\begin{aligned}
\mathcal{F}(r,\Gamma_k)=\min_{y_g} \quad & \mathbf{R}(z_g) \\[2pt] 
\text{s.t.} \quad & f(y_g) = 0, \\[2pt] 
& {\|v_1^{(i)}\|}_2^2 = 1, \quad i = 1,2,\dots,r, \\[2pt] 
& v_1^{(i)\top} v_1^{(j)} = 0, \quad i \neq j. \\[2pt] 
\end{aligned}
\end{equation}

{\it The outer optimization}:~By evaluating the minimum cost $\mathcal{F}(r,\Gamma_k)$ obtained from the inner optimization~\eqref{INNER}, we solve the outer problem which identifies the overall optimal configuration by searching over all feasible  $m \ge r \ge p$ and all admissible integer partitions $\Gamma_k=(k_1,k_2,\ldots,k_r)$ satisfying $k_1 \ge k_2 \ge \dots \ge k_r$ and $\sum_{i=1}^r k_i = m$. The resulting outer optimization problem is formulated as
\begin{equation}
\label{OUTER}
\begin{aligned}
\min_{r, \Gamma_k} \quad & \mathcal{F}(r, \Gamma_k) \\[2pt]
\text{s.t.} \quad & p \le r \le m, \\[2pt]
& \sum_{i=1}^r k_i = m, \\[2pt]
& k_1 \ge k_2 \ge \dots \ge k_r \ge 1.
\end{aligned}
\end{equation}
The overall problem can then be presented as a comprehensive nested optimization problem given as:
\begin{equation}
\label{eq:bilevel_optimization_final}
\begin{aligned}
& \min_{r, \Gamma_k} \quad
  \underbrace{
  \left\{
  \begin{aligned}
    \min_{y_g} \quad & \mathbf{R}(z_g) \\[2pt]
    \text{s.t.} \quad & f(y_g) = 0, \\[2pt] 
    & \|v_1^{(i)}\|_2^2 = 1, \quad i = 1,2, \dots, r, \\[2pt]
    & v_1^{(i)\top} v_1^{(j)} = 0, \quad i \ne j.
  \end{aligned}
  \right\}
  }_{\text{INNER Optimization}} \\[12pt]
& \underbrace{
  \text{s.t.} \quad p \le r \le m, \quad \sum_{i=1}^r k_i = m, \quad k_1 \ge k_2 \ge \dots \ge k_r \ge 1.
  }_{\text{OUTER Optimization}}
\end{aligned}
\end{equation}
Suppose that the optimal solution of~\eqref{eq:bilevel_optimization_final} is obtained as
\begin{equation}
\label{eq:y_g_cases}
y_1^\star = 
\begin{bmatrix}
v_1^{(1)\star} \\
v_2^{(1)\star}\\ 
\vdots \\
v_{k_r}^{(r)\star}\\[3pt]
z_1^\star \\
\mu^\star
\end{bmatrix} \quad (\text{for } g=1),
\qquad \text{and} \qquad
y_2^\star = 
\begin{bmatrix}
v_1^{(1)\star} \\
v_2^{(1)\star}\\ 
\vdots \\
v_{k_r}^{(r)\star}\\[3pt]
z_2^\star 
\end{bmatrix} \quad (\text{for } g=2).
\end{equation}
By reconstructing the parameter vector $z_g^\star$ via $\operatorname{vec}(\Delta^\star_g) = D z_g^\star$ into its respective matrix $\Delta_g^\star \in \mathbb{R}^{n \times n}$, the updated matrices can be given as
\begin{equation}
\label{eq:optimal_recovery}
    Y^\star_\mu = X + \Delta^\star_1 \quad (\text{for }g = 1), \qquad
Y^\star_\lambda = X + \Delta^\star_2 \quad (\text{for }g = 2).
\end{equation}
Similarly, after extracting the vector components $v_j^{(i)\star}$ from $y_1^\star$ and $y_2^\star$, we obtain the vectors associated with the target eigenvalue - either the optimized $\mu^\star$ (for $g=1$) or the fixed $\lambda$ (for $g=2$). By construction, the updated matrix $Y^\star$ successfully guarantees the desired lower bounds on the geometric and algebraic multiplicities for both cases
\begin{equation}
\label{eq:multiplicity_bounds1}
\begin{aligned}
\mathrm{GM}(\mu^\star;Y^\star_\mu) \ge p, \qquad \mathrm{AM}(\mu^\star;Y^\star_\mu) \ge m \qquad &(\text{for } g=1)
\end{aligned}
\end{equation}
and
\begin{equation}
\label{eq:multiplicity_bounds2}
\begin{aligned}
\mathrm{GM}(\lambda;Y^\star_\lambda) \ge p, \qquad \mathrm{AM}(\lambda;Y^\star_\lambda) \ge m \qquad &(\text{for } g=2).
\end{aligned}
\end{equation}

\subsection{Nature of the optimization problem}
\label{subsec:optimization_topology}

The inner optimization problem in \eqref{eq:bilevel_optimization_final} possesses a convex quadratic objective function of the form
\begin{equation}
\mathbf{R}(z_g)=\|Dz_g\|_2^2
=
z_g^\top(D^\top D)z_g,
\end{equation}
where the structural matrix $D$ has full column rank. Consequently, $D^\top D$ is symmetric positive definite, implying that $\mathbf{R}(z_g)$ is convex with respect to the structural parameters $z_g$. However, the complete optimization problem is formulated over the variable $y_g$, which additionally contains the spectral variables associated with the eigenvalue and generalized eigenvectors. 
The nonlinear feasibility condition
\[
f(y_g)=0
\]
introduces bilinear coupling between the perturbation parameters and spectral variables through the Jordan chain relations. Furthermore, the normalization and orthogonality constraints
\[
{\|v_1^{(i)}\|}_2^2=1
\quad \text{and}\quad
v_1^{(i)\top}v_1^{(j)}=0
\]
are quadratic equality constraints  that define a non-convex feasible set.

Therefore, although the objective function is convex with respect to the structural parameters $z_g$, the overall constrained optimization problem in the augmented variable $y_g$ is inherently non-convex and may admit multiple local minima.

\subsection{Numerical solution of the nested optimization problem}
\label{remark:numerical_solution}

The nested optimization problem presented in \eqref{eq:bilevel_optimization_final} features a two-level structure comprising a continuous nonlinear inner problem and a discrete combinatorial outer problem \cite{dempe2002foundations}. These two problems can be solved using distinct numerical procedures.

\subsubsection{The inner optimization}\label{subsec:inner} 
To solve the inner continuous optimization problem numerically, we use MATLAB's \texttt{fmincon} function. Specifically, we employ the Sequential Quadratic Programming (SQP) algorithm \cite{nocedal2006numerical}. SQP is well-suited for handling the nonlinear equality constraints associated with the structured feasibility system and ensuring convergence to a local minimum.

\subsubsection{The outer optimization}\label{subsec:outer} 
Among all the local optimal solutions yielded by the inner optimization for each valid configuration, the outer problem identifies the overall optimum by comparing their minimized objective values. The outer problem defined in \eqref{OUTER} is a discrete integer programming problem. The equality $\sum_{i=1}^r k_i = m$ and the ordering $k_1 \ge k_2 \ge \dots \ge k_r \ge 1$ formally define the problem of finding the integer partitions of $m$ into $r$ parts \cite{orucc2016number}, bounded by $m \ge r \ge p$. In our context, each partition $\Gamma_k=(k_1,k_2,\ldots,k_r)$ represents a candidate Jordan structure consisting of $r$ chains of lengths $k_i$. 
For each admissible $r\in\{p,p+1,\ldots,m\}$, we computationally generate these partitions using a recursive procedure. At first, $k_1$ is selected from the set $\{1,2,\ldots,m-r+1\},$ after which the remainder $m-k_1$ is recursively partitioned into $r-1$ positive integers. Each generated partition is subsequently sorted in nonincreasing order and duplicates are removed. Repeating this procedure for all bounded values of $r$ produces the complete, finite collection of admissible configurations considered by the outer optimization. This procedure is given in Algorithm~\ref{alg:partition_generation}.
\begin{algorithm}[H]
\caption{Generation of admissible configurations}
\label{alg:partition_generation}
\begin{algorithmic}[1]

\REQUIRE \ Positive integers $m,p$ and $r$ with $p\le r\le m$

\ENSURE All distinct partitions $(k_1, k_2,\ldots,k_r)$ satisfying
$k_1\ge k_2\ge\cdots\ge k_r\ge1$ and
$k_1+k_2+\cdots+k_r=m$

\STATE Initialize $\Gamma \gets\emptyset$

\STATE Recursively generate all compositions of $m$ into exactly $r$ positive integers (using steps 3-9)

\FOR{$k_1=1,\ldots,m-r+1$}

    \STATE Recursively generate all compositions $(k_2,k_3,\ldots,k_r)$ of $m-k_1$ into $r-1$ positive integers (using step 2)

    \FORALL{generated tuples $(k_1,k_2,\ldots,k_r)$}

        \STATE Sort $(k_1,k_2,\ldots,k_r)$ into nonincreasing order
        $(\hat{k}_1,\hat{k}_2,\ldots,\hat{k}_r)$ with
        $\hat{k}_1\ge\hat{k}_2\ge\cdots\ge\hat{k}_r\ge1$

        \STATE Insert $(\hat{k}_1,\hat{k}_2,\ldots,\hat{k}_r)$ into $\Gamma$

    \ENDFOR

\ENDFOR

\STATE Remove duplicate elements from $\Gamma$

\RETURN $\Gamma$

\end{algorithmic}
\end{algorithm}

For example, when $m=5$ and $r=3$, the only admissible partitions are
\[
(3,1,1)
\qquad \text{and} \qquad
(2,2,1),
\]
which correspond exactly to the two possible Jordan chain configurations.

Let $\Gamma_{m,r}$ denote the set of all partitions of $m$ into exactly $r$ positive integers. For fixed $m$ and $r$, the number of admissible configurations is exactly the number of partitions of $m$ into $r$ positive integers, say $s_r(m)$. That is, $|\Gamma_{m,r}|=s_r(m)$. When $r\in\{p,p+1,\ldots,m\}$, the total number of configurations the outer loop explores is:
\begin{equation}
    N=\sum_{r=p}^{m} s_r(m).
\end{equation}
Since
\(\sum_{r=1}^{m} s_r(m)=p(m),\)
where $p(m)$ denotes the ordinary partition function, it follows that
\(N \le p(m).\)
Therefore, the total search space is bounded by the number of partitions of $m$. Using the Hardy--Ramanujan asymptotic formula \cite{orucc2016number},
\(p(m)\sim
\frac{1}{4m\sqrt{3}}
e^{\pi\sqrt{\frac{2m}{3}}},\)
the search space grows only as
\(p(m)=e^{O(\sqrt{m})}.
\)
Hence, the number of admissible configurations grows subexponentially with respect to $m$, which guarantees that this exhaustive search remains computationally tractable for the problem sizes considered in this work.
Additionally, for each fixed $r<m$, the sharper upper bound of Oru\c{c}~\cite{orucc2016number} yields
\[
s_r(m)
=
p(m,r)
\le
\frac{5.44}{m-r}
e^{\pi\sqrt{\frac{2(m-r)}{3}}}.
\]

The overall numerical algorithm thus operates by iteratively solving continuous quadratic subproblems via SQP (the inner loop) across the exhaustively enumerated set of integer partitions (the outer loop). As a result, we obtain an approximate local solution to the nearest matrix problem~\eqref{eq:prob1} or~\eqref{eq:prob2}.


\subsubsection{Algorithm and Initialization Strategy}
We summarized the procedure explained in subsections~\ref{subsec:inner} and~\ref{subsec:outer} to solve the optimization problem~\eqref{eq:bilevel_optimization_final} in Algorithm~\ref{alg:nearest_structured_AM_GM}. Since the inner optimization problem is nonconvex due to the nonlinear eigenvalue constraints and the coupling among the variables $(v_j^{(i)}, z,\mu)$, the solution obtained by the SQP algorithm is sensitive to the initial guess \cite{nocedal2006numerical}. To reduce this dependence on the initial guess, we employ the following initialization strategy.

The spectral variables, namely the unknown eigenvalue $\mu$ and the generalized eigenvectors ${v_j^{(i)}}$, are initialized randomly using samples drawn from the standard normal distribution $\mathcal N(0,1)$. In contrast, the structural parameters are initialized at the origin, that is, $z^{(0)}=\mathbf{0}\in\mathbb{R}^l$. As a result, the initial perturbation satisfies $\Delta=\mathbf{0}$, and the optimization starts from the original matrix $X$.

For each admissible structural configuration, the algorithm is run 20 times with different random initializations of the spectral variables. Among all feasible solutions obtained, we select the one with the smallest perturbation norm for that configuration. Finally, among all configurations considered, the solution with the smallest computed perturbation norm ${\|\Delta\|}_F$ is chosen as the final computed solution.
\begin{algorithm}[h!]
\caption{Nearest matrix enforcing $\mathrm{GM}(\mu;Y)\ge p$ and $\mathrm{AM}(\mu;Y)\ge m$. }
\label{alg:nearest_structured_AM_GM}

\begin{algorithmic}[1]

\STATE \textbf{Input:} 
$X\in\mathbb R^{n\times n}$, transformation matrix $D\in\mathbb R^{n^2\times l}$, integers $1\le p\le m\le n$.

\FOR{$r=p,p+1,\ldots,m$}

    \STATE Generate all integer partitions
    $\Gamma=(k_1,k_2,\ldots,k_r)$ satisfying
    \[
    \sum_{i=1}^{r} k_i = m,
    \qquad
    k_1 \ge k_2 \ge \cdots \ge k_r \ge 1.
    \]
\FOR{each admissible configuration $(r,\Gamma)$}
        
        \FOR{$t = 1$ \TO $20$ (Multi-start trials)}

            \STATE \textbf{Decision variables:} 
            $z\in\mathbb R^l$, scalar $\mu$, and vectors $v_j^{(i)}\in\mathbb R^n$ for $i=1,2,\dots,r$ and $j=1,2,\dots,k_i$.
            
            \STATE \textbf{Initialize:} Set structural parameters $z^{(0)} = \mathbf{0}$. Initialize $\mu^{(0)}$ and $\{v_j^{(i)}\}^{(0)}$ via standard normal distribution $\mathcal N(0,1)$.

            \STATE Solve \eqref{eq:bilevel_optimization_final} using the SQP implementation of MATLAB's \texttt{fmincon} solver. 
            \STATE Extract the optimal structural parameter vector $z_t^\star$ and the corresponding objective function value $\mathbf{R}_t^\star$.
        \ENDFOR
        
        \STATE Determine the minimal objective over the 20 trials for the current configuration $(r, \Gamma)$:
        \[
        \mathbf{R}_{(r,\Gamma)}^\star = \min_{1 \le t \le 20} \mathbf{R}_t^\star
        \]
        \STATE Let $z_{(r,\Gamma)}^\star$ be the parameter vector corresponding to $\mathbf{R}_{(r,\Gamma)}^\star$.
    \ENDFOR
\ENDFOR

\STATE Extract the minimum objective over all evaluated multiplicities $r$ and partitions $\Gamma$:
\[
\mathbf{R}^\star = \min_{r, \Gamma} \mathbf{R}_{(r,\Gamma)}^\star
\]
\STATE Let $z^\star$ be the optimal structural parameter vector corresponding to $\mathbf{R}^\star$.

\STATE Construct the optimal structured perturbation matrix $\Delta^\star = \operatorname{mat}(Dz^\star)$.
\STATE Set $Y^\star=X+\Delta^\star$.

\STATE \textbf{Output:} $\Delta^\star$, $Y^\star$, and $\mu^\star$.

\end{algorithmic}
\end{algorithm}

\begin{remark}
\label{rem:known_eigenvalue}
Algorithm~\ref{alg:nearest_structured_AM_GM} is formulated for the unknown-eigenvalue case (problem~\eqref{eq:prob1}), where the target eigenvalue $\mu$ is treated as an optimization variable and is computed together with the perturbation and generalized eigenvectors. 
For the known-eigenvalue case (problem~\eqref{eq:prob2}), the prescribed eigenvalue $\lambda$ is fixed in advance and therefore does not need to be optimized. In this setting, $\mu$ is simply replaced by $\lambda$ in all Jordan chain constraints, and the scalar variable $\mu$ is removed from both the decision-variable vector and the initialization procedure. Apart from this modification, the overall algorithm remains unchanged. 
\end{remark}

\section{Numerical Experiments}
\label{sec numerical}

In this section, we present numerical experiments to demonstrate the effectiveness of the proposed optimization framework. 
%

All computations were carried out in \textsc{MATLAB} [R2024a] version using the \texttt{fmincon} solver 
with the Sequential Quadratic Programming (SQP) algorithm to solve the inner optimization problem~\eqref{eq:bilevel_optimization_final}. The simulations were performed on a standard desktop computing environment equipped with an Intel Core i3 processor and 8 GB of RAM.

\begin{example}[Nearest Toeplitz matrix: known vs. unknown eigenvalue]
\label{ex:toeplitz_cases}
Consider the following $6 \times 6$ Toeplitz matrix
\[
X =
\begin{bmatrix}
 1.8 & -0.7 &  0.4 & -0.2 &  0.1 &  0.0 \\
-0.5 &  1.8 & -0.7 &  0.4 & -0.2 &  0.1 \\
 0.3 & -0.5 &  1.8 & -0.7 &  0.4 & -0.2 \\
-0.1 &  0.3 & -0.5 &  1.8 & -0.7 &  0.4 \\
 0.0 & -0.1 &  0.3 & -0.5 &  1.8 & -0.7 \\
 0.2 &  0.0 & -0.1 &  0.3 & -0.5 &  1.8
\end{bmatrix}
\]
with eigenvalues    $ 3.4538 + 0.0000i$,
  $2.4664 + 0.0000i$,
   $1.0835 + 0.0673i$,
   $1.0835 - 0.0673i$, 
   $1.3564 + 0.0716i$, and 
   $1.3564 - 0.0716i$.
There exists a transformation matrix $D \in \mathbb R^{n^2 \times (2n-1)}$, where $n=6$ such that any Toeplitz matrix $\Delta \in \mathbb R^{6 \times 6}$ satisfies $\operatorname{vec}(\Delta) = Dz$, $z \in \mathbb R^{11}.$

\medskip
\noindent\textbf{ Unknown eigenvalue (Problem~\eqref{eq:prob1} with $p=2$ and $m=3$)}

We use Algorithm~\ref{alg:nearest_structured_AM_GM} to find the nearest Toeplitz matrix $Y_\mu = X + \Delta_1$ such that $Y_\mu$ admits an eigenvalue $\mu \in \mathbb R$ satisfying
\[
\mathrm{GM}(\mu;Y_\mu) \ge 2 \quad \text{and}\quad \mathrm{AM}(\mu;Y_\mu) \ge 3.
\label{eq:toeplitz_mult_unknown}
\]
Here, $\mu$ is treated as an additional variable. 
The algorithm terminated with $\texttt{exitflag} = 1$, indicating its convergence to a stationary point.
%
%
The optimal Toeplitz matrix $\hat Y_{\mu} = X + \hat \Delta_1$ is obtained as
\[
\hat Y_{\mu} =
\begin{bmatrix}
 1.8003 & -0.6083 &  0.5027 & -0.2810 & -0.0287 &  0.3765 \\
-0.4760 &  1.8003 & -0.6083 &  0.5027 & -0.2810 & -0.0287 \\
 0.2893 & -0.4760 &  1.8003 & -0.6083 &  0.5027 & -0.2810 \\
-0.0747 &  0.2893 & -0.4760 &  1.8003 & -0.6083 &  0.5027 \\
-0.1264 & -0.0747 &  0.2893 & -0.4760 &  1.8003 & -0.6083 \\
 0.2802 & -0.1264 & -0.0747 &  0.2893 & -0.4760 &  1.8003
\end{bmatrix}
\]
with an eigenvalue $\mu^\star = 1.2067$. 
The optimal perturbation satisfies, $\|\hat Y_{\mu} - X\|_F=0.5679$
such that $\mu^\star = 1.2067$ is an eigenvalue of $\hat Y_{\mu}$ with
$\mathrm{AM}(\mu^\star;\hat Y_\mu) = 4$ and $\mathrm{GM}(\mu^\star;\hat Y_\mu) =4$.


\medskip
\noindent\textbf{Known eigenvalue (Problem~\eqref{eq:prob2} with $p=2$ and $m=3$)}

Now, suppose the target eigenvalue is prescribed in advance. We take the target eigenvalue as $\lambda = 1.8$. 
We apply Algorithm~\ref{alg:nearest_structured_AM_GM} with $\lambda$ prescribed as $\lambda=1.8$ to find the nearest Toeplitz perturbation $Y_\lambda = X + \Delta_2$ satisfying
\[
\mathrm{GM}(\lambda;Y_\lambda) \ge 2 \quad \text{and} \quad \mathrm{AM}(\lambda;Y_\lambda) \ge 3.
\label{eq:toeplitz_mult_known}
\]
By fixing $\lambda$, the decision variable space is reduced, and the problem restricts the feasible set to $\mathcal{V}_\lambda$ defined in~\eqref{U Set}. 
The algorithm terminated with $\texttt{exitflag} = 1$.

Since the eigenvalue is constrained to a specific non-optimal location relative to the unknown eigenvalue case, the required perturbation norm strictly increases as $\|\hat Y_\lambda - X\|_F = 1.2671$. The optimal 
Toeplitz matrix $\hat Y_{\lambda} = X + \hat \Delta_2$ is obtained as
\[
\hat Y_{\lambda} =
\begin{bmatrix}
 1.9746 & -0.4025 &  0.2627 & -0.4592 &  0.2597 &  0.0697 \\
-0.3471 &  1.9746 & -0.4025 &  0.2627 & -0.4592 &  0.2597 \\
 0.0986 & -0.3471 &  1.9746 & -0.4025 &  0.2627 & -0.4592 \\
-0.3088 &  0.0986 & -0.3471 &  1.9746 & -0.4025 &  0.2627 \\
 0.0341 & -0.3088 &  0.0986 & -0.3471 &  1.9746 & -0.4025 \\
-0.2849 &  0.0341 & -0.3088 &  0.0986 & -0.3471 &  1.9746
\end{bmatrix}.
\]
%
As in the unknown eigenvalue case, direct eigenvalue computation and subsequent verification confirm that $\mathrm{AM}(\lambda;\hat Y_\lambda) = 4$ and $\mathrm{GM}(\lambda;\hat Y_\lambda) = 3$, successfully enforcing the required multiplicity bounds at the prescribed location of $\lambda = 1.8000$ while preserving the structure of the matrix. 

In this case, the optimal matrix $\hat Y_\lambda$ is a defective matrix. This implies that the distance
$\|\hat Y_\lambda - X\|_F = 1.2671$ gives an upper bound to the structured distance to a defective matrix for $X$. To improve the estimation for the distance to the nearest structured defective matrix for $X$, we compute the nearest matrix in Problem~\eqref{eq:prob2} for the Toeplitz matrix $X$ with the prescribed eigenvalue $\lambda=1.8$, while varying the multiplicity bounds $m$ and $p$. The results are depicted in Table~\ref{tab:multiplicity_bounds}. The third column records the norm of the optimal perturbation achieved by Algorithm~\ref{alg:nearest_structured_AM_GM} for the multiplicity bounds $m$ (first column) and $p$ (second column). The forth and fifth columns respectively record the $\mathrm{AM}(\lambda;\hat Y_\lambda)$ and $\mathrm{GM}(\lambda;\hat Y_\lambda)$. Thus the minimum of the third column, i.e. ${\|\hat \Delta\|}_F=0.5774$ gives a good estimation to the structured distance to defectivity for $X$. 



\begin{table}[ht]
\centering
\small
\setlength{\tabcolsep}{4pt}
\begin{tabular}{ccccc}
\hline
 \textbf{Target $m$} & \textbf{Target $p$} & $\mathbf{||\hat \Delta||_F}$ & \textbf{Final AM} & \textbf{Final GM} \\
\hline
 2 & 1 & 0.6839 & 2 & 1 \\
 2 & 2 & 0.7958 & 2 & 2 \\
 3 & 1 & 0.5774 & 3 & 1 \\
 3 & 2 & 1.2671 & 4 & 3 \\
 3 & 3 & 1.4970 & 5 & 5 \\
 4 & 1 & 2.0322 & 6 & 2 \\
 4 & 2 & 2.0322 & 6 & 2 \\
 4 & 3 & 2.2159 & 6 & 6 \\
 4 & 4 & 2.2159 & 6 & 6 \\
\hline
\end{tabular}
\caption{Solution to Problem~\eqref{eq:prob2} with varying multiplicity bounds for the Toeplitz matrix $X$ in Example~\ref{ex:toeplitz_cases}}
\label{tab:multiplicity_bounds}
\end{table}

The left plot in Figure~\ref{fig:eigenvalue_comparison} shows the unperturbed eigenvalues of the matrix $X$. The right plot traces the movement of the eigenvalues (surrounded by circles) of $X$ with respect to perturbation $X+t \hat \Delta_2$ as $t$ moves from $0$ to $1$, $\hat \Delta_2$ being the minimal Toeplitz perturbation such that $\hat Y_\lambda=X+\hat \Delta_2$ has the prescribed eigenvalue $\lambda=1.8$ (surrounded by a diamond) with $\mathrm{AM}(\lambda;\hat Y_\lambda) = 4$ and $\mathrm{GM}(\lambda;\hat Y_\lambda) = 3$. The eigenvalue curves originated from the unperturbed eigenvalues $2.4664 + 0.0000i$, $1.0835 + 0.0673i$,  $1.3564 + 0.0716i$, and $1.3564 - 0.0716i$ meet to make the target eigenvalue $\lambda =1.8$, an eigenvalue of $X+\hat \Delta$ with algebraic multiplicity $4$.

\begin{figure}[ht]
\centering
\begin{subfigure}{0.50\textwidth}
    \centering
    \includegraphics[width=\linewidth]{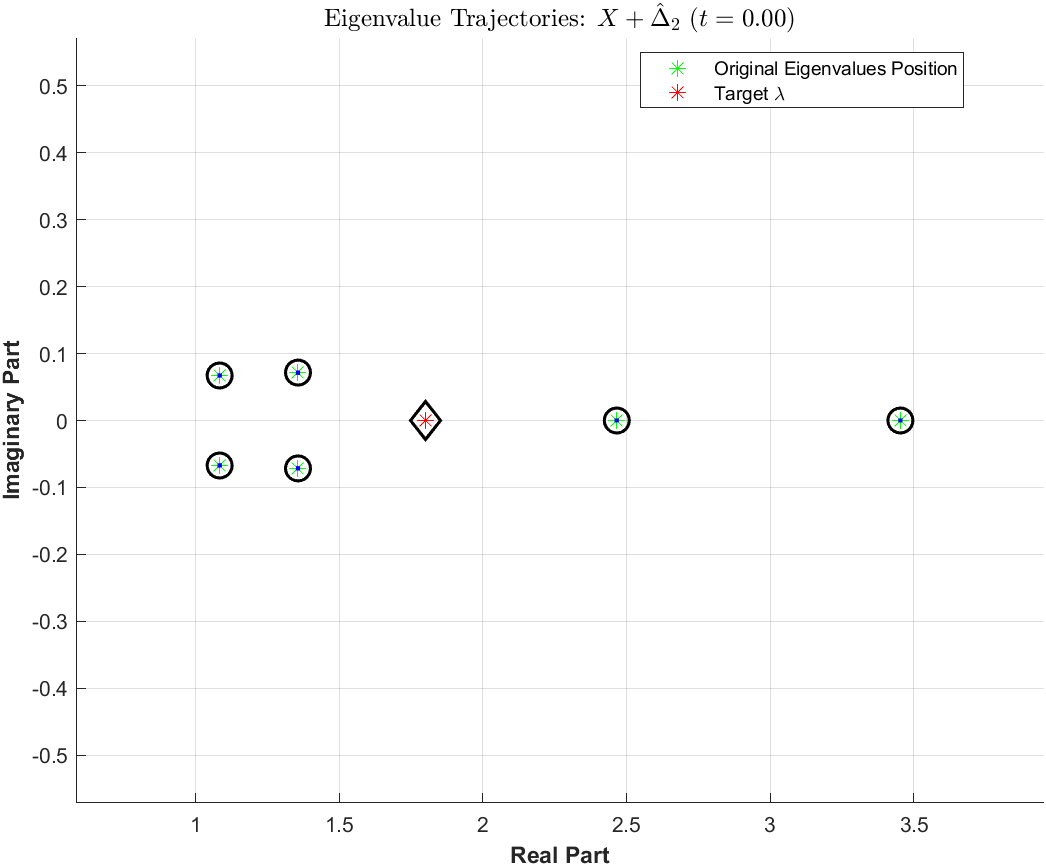}
    \caption{$t = 0.00$}
    \label{fig:T0}
\end{subfigure}\hfill
\begin{subfigure}{0.50\textwidth}
    \centering
     \includegraphics[width=\linewidth]{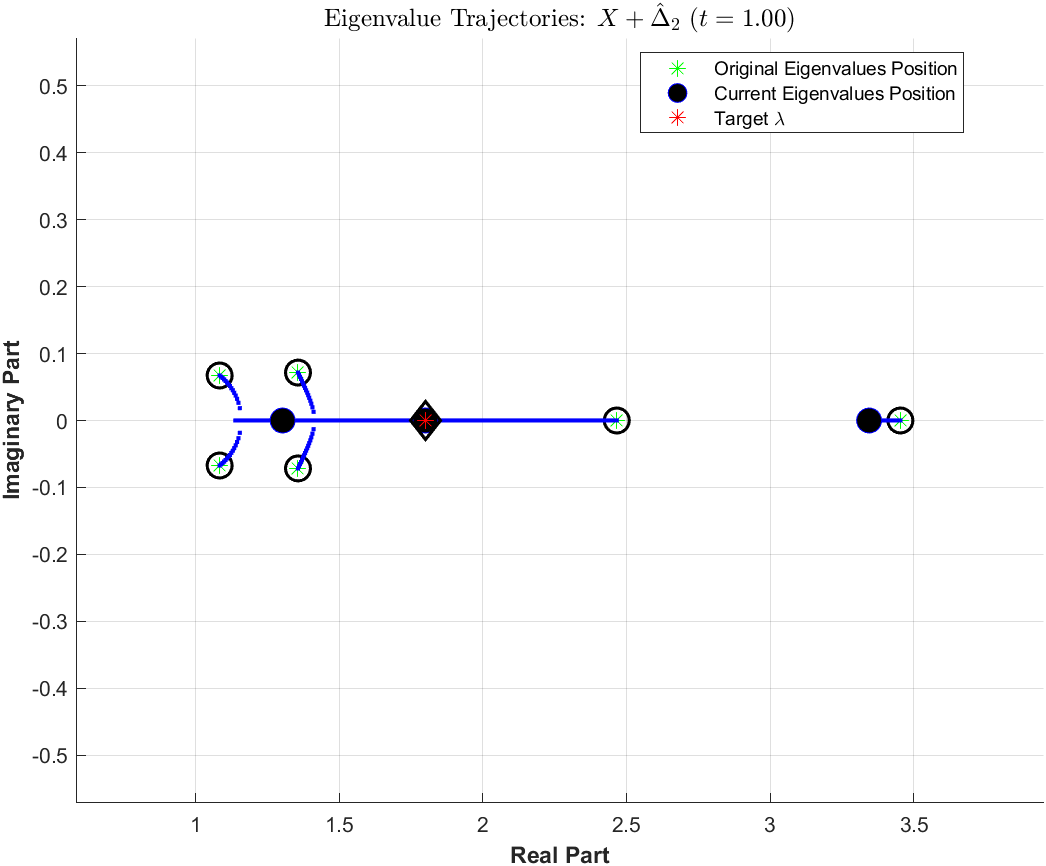}
    \caption{$t = 1.00$}
    \label{fig:T1}
\end{subfigure}
\caption{Eigenvalue perturbation curves for the matrix $Y(t) = X + t \hat\Delta_2$ in Example~\ref{ex:toeplitz_cases}. The left panel (a) shows the original eigenvalues before the perturbation ($t=0$). The right panel (b) shows the eigenvalue curves (dotted blue lines) and the final eigenvalues at $t=1$. Multiple eigenvalue curves meet exactly at the target $\lambda$ (surrounded by a diamond).}
\label{fig:eigenvalue_comparison}
\end{figure}

\end{example}

\begin{example}[Nearest Hankel matrix: known vs. unknown eigenvalue]
\label{ex:hankel_cases}

Consider the following $5 \times 5$ Hankel matrix:
$$
X=
\begin{bmatrix}
 1.2 & -0.4 &  0.7 &  1.1 & -0.3\\
-0.4 &  0.7 &  1.1 & -0.3 &  0.9\\
 0.7 &  1.1 & -0.3 &  0.9 & -1.5\\
 1.1 & -0.3 &  0.9 & -1.5 &  0.8\\
-0.3 &  0.9 & -1.5 &  0.8 &  0.6
\end{bmatrix}
$$
with five distinct real eigenvalues. Then there exists the transformation matrix $D \in \mathbb R^{25 \times 9}$ such that any Hankel matrix $\Delta \in \mathbb R^{5\times 5}$ satisfies
 $\operatorname{vec}(\Delta)=Dz$, $z\in\mathbb R^{9}$.
Since any real Hankel matrix is symmetric, it is guaranteed to be diagonalizable. Therefore, defective eigenvalues cannot occur, and the geometric multiplicity must exactly equal the algebraic multiplicity for each eigenvalue of $X$.


\medskip
\noindent\textbf{Unknown eigenvalue (Problem~\eqref{eq:prob1} with $p=2$ and $m=3$)}

We use Algorithm~\ref{alg:nearest_structured_AM_GM} to find the nearest Hankel perturbation $Y_\mu = X + \Delta_1$ such that $Y_\mu$ admits an unknown eigenvalue $\mu \in \mathbb R$ satisfying
\[\mathrm{GM}(\mu;Y_\mu) \ge 2 \quad \text{and} \quad  \mathrm{AM}(\mu;Y_\mu) \ge 3.
\label{eq:hankel_mult_unknown}\]
The optimization problem was solved using the proposed nested approach as given in \eqref{eq:bilevel_optimization_final}, treating $\mu$ as an additional variable. The algorithm terminated with $\texttt{exitflag} = 1$, indicating convergence to a stationary point.


The optimal Hankel matrix $\hat Y_{\mu} = X + \hat \Delta_1$ is obtained as
$$
\hat Y_{\mu}=
\begin{bmatrix}
 0.5230 & -0.6780 &  0.3965 &  1.2827 & -0.0656 \\
-0.6780 &  0.3965 &  1.2827 & -0.0656 &  0.8119 \\
 0.3965 &  1.2827 & -0.0656 &  0.8119 & -1.2025 \\
 1.2827 & -0.0656 &  0.8119 & -1.2025 &  1.1737 \\
-0.0656 &  0.8119 & -1.2025 &  1.1737 &  0.5227
\end{bmatrix}
$$
with an eigenvalue $\mu^\star = 1.5542$. The optimal perturbation satisfies $\|\hat Y_\mu- X\|_F = 1.3709$ such that $\mu^\star = 1.5542$ is an eigenvalue of $\hat Y_\mu$ with 
$\mathrm{AM}(\mu^\star;\hat Y_\mu) = \mathrm{GM}(\mu^\star;\hat Y_\mu) = 3$, successfully satisfying the lower bounds $m=3$ and $p=2$ while preserving the structure of the matrix.

\medskip
\noindent\textbf{Known eigenvalue (Problem~\eqref{eq:prob2} with $p=2$ and $m=3$)}

Now, suppose the target eigenvalue is prescribed in advance as $\lambda = 0.47$. We again apply Algorithm~\ref{alg:nearest_structured_AM_GM} with prescribed $\lambda=0.47$ to find the nearest Hankel perturbation $Y_\lambda = X + \Delta_2$ satisfying
\[
\mathrm{GM}(\lambda;Y_\lambda) \ge 2 \quad \text{and} \quad  \mathrm{AM}(\lambda;Y_\lambda) \ge 3.
\]
The optimization problem was solved using the proposed nested approach as given in \eqref{eq:bilevel_optimization_final}. The solver terminated with $\texttt{exitflag} = 1$.

Since the eigenvalue is constrained to a non-optimal location relative to the unconstrained minimum ($\mu^\star = 1.5542$ from unknown eigenvalue case), the required perturbation size strictly increases as $\|\hat Y_\lambda - X\|_F = 2.7811$. 
The optimal Hankel matrix $\hat Y_{\lambda} = X + \hat \Delta_2$ is obtained as
$$
\hat Y_{\lambda}=
\begin{bmatrix}
 0.1134 & -0.1134 & -0.0127 &  0.5437 & -0.2021 \\
-0.1134 & -0.0127 &  0.5437 & -0.2021 &  0.4793 \\
-0.0127 &  0.5437 & -0.2021 &  0.4793 & -0.6737 \\
 0.5437 & -0.2021 &  0.4793 & -0.6737 &  0.7644 \\
-0.2021 &  0.4793 & -0.6737 &  0.7644 & -0.3059
\end{bmatrix}.
$$
As in the unknown eigenvalue case, direct eigenvalue computation shows three identical eigenvalues at $\lambda = 0.4700$, successfully satisfying the lower bounds of $m=3$ and $p=2$ while preserving the Hankel structure of the matrix.

\end{example}

\begin{example}
\label{ex:comparison_literature}{(Comparison with Mengi \cite{mengi2011locating})}

In this example, we compare our proposed method with the spectral norm (2-norm) optimization methods developed by Armentia et al.~\cite{armentia2020nearest} and Mengi~\cite{mengi2011locating}.  Their techniques use generalized Malyshev-type formulas to determine the distance by minimizing the supremum of a specific singular value within an augmented block matrix. 
Our approach diverges from these methods in two primary ways. First, their techniques search strictly within the unconstrained space $\mathbb{R}^{n \times n}$, meaning they cannot explicitly preserve physical matrix structures. Second, their objective function minimizes the spectral norm, whereas ours uses the Frobenius norm.
Specifically, such approaches characterize the Wilkinson distance to defectivity via
\[
W_2(X) =
\inf_{\lambda \in \mathbb{R}}
\sup_{\gamma \in (0,1]}
\sigma_{2n-1}
\left(
\begin{bmatrix}
X - \lambda I & \gamma I \\
0 & X - \lambda I
\end{bmatrix}
\right),
\]
which represents the smallest spectral-norm perturbation required to render the matrix $X$ defective.

Consider the non-symmetric upper Hessenberg matrix $X$~\cite{mengi2011locating}:
$$
X = \begin{bmatrix}
 3 & -2 &  1 &  4 \\
-1 & -3 &  1 &  1 \\
 0 & -4 &  2 &  1 \\
 0 &  0 &  5 &  1
\end{bmatrix}.
$$
We seek the nearest matrix $Y = X + \Delta$ possessing an eigenvalue $\mu$ satisfying
$$
\mathrm{AM}(\mu;Y) \ge 2 \quad \text{and} \quad \mathrm{GM}(\mu;Y) \ge 1.
$$

\medskip

\noindent
\textbf{Case 1: Singular value minimization~(\cite{mengi2011locating})}

Applying the singular value optimization framework in~\cite{mengi2011locating} yields an optimal perturbation $\Delta_{SVD}$ such that $Y_{SVD} = X + \Delta_{SVD}$ admits an eigenvalue $\mu^\star = 1.5181$ with $\mathrm{AM}(\mu^*;Y_{SVD}) = 2$ and $\mathrm{GM}(\mu^*;Y_{SVD}) = 1$. The resulting matrix $Y_{\mathrm{SVD}}$ reported in \cite{mengi2011locating} is
$$
Y_{SVD} = \begin{bmatrix}
 3.0082 & -2.0306 &  1.0215 &  3.9132 \\
-0.5849 & -3.0152 &  0.9122 &  1.0544 \\
-0.3014 & -3.8648 &  1.9683 &  1.3209 \\
 0.1332 &  0.0947 &  4.8953 &  1.3066
\end{bmatrix}.
$$
%
The spectral norm of $\Delta_{SVD}$ is exactly the Wilkinson distance 
$\|\Delta_{SVD}\|_2 = 0.5556$. The Frobenius norm of $\Delta_{SVD}$ is evaluated as $\|\Delta_{SVD}\|_F = 0.7330$.

Although $Y_{SVD}$ mathematically satisfies the prescribed multiplicity bounds for the optimized eigenvalue $\mu^\star = 1.5181$, the perturbation is computed over the full unconstrained matrix space $\mathbb{R}^{4 \times 4}$ and fails to preserve the Hessenberg structure (zeros below the first subdiagonal) of $X$. 

\medskip

\noindent
\textbf{Case 2: Proposed Optimization Framework (POF) }

To demonstrate the flexibility of the proposed method, we 
solve Problem~\eqref{eq:prob2} (known eigenvalue) and Problem~\eqref{eq:prob1} (unknown eigenvalue) for matrix $X$ with $m=2$ and $p=1$. For each problem, we consider both an unconstrained search space (to provide a direct comparison with Case 1) and a structured search space (preserving the Hessenberg structure of $X$). We solved all these problems using the proposed nested optimization approach as given in \eqref{eq:bilevel_optimization_final}.

\medskip
\noindent
\textit{Sub-case 2.1:~Known eigenvalue (Problem~\eqref{eq:prob2} with $\lambda= 1.5181$)}

We fix the target eigenvalue to $\lambda = 1.5181$, identically matching the classical optimum found in Case 1.

\textbf{(a) Unstructured search space:} We define the transformation matrix as the identity matrix $D = I_{16}$, parameterizing the full space $\operatorname{vec}(\Delta) = I_{16}z$. 
This results in the perturbed matrix $Y_{\text{POF}}^{(1)} = X + \Delta_1$ given by
$$
Y_{\text{POF}}^{(1)} = \begin{bmatrix}
3.0444 & -2.0113 & 0.9981 & 3.9777 \\
-0.7001 & -3.0765 & 0.9869 & 0.8492 \\
-0.3756 & -3.9043 & 2.0164 & 1.1888 \\
-0.0303 & 0.0077 & 5.0013 & 1.0152
\end{bmatrix}.
$$
The Frobenius norm is significantly reduced compared to the SVD approach, yielding 
$$
{\|Y_{\text{POF}}^{(1)} - X\|}_F =  \|\Delta_1\|_F = 0.5556.
$$ 
Note that the spectral norm of $\Delta_1$ is $\|\Delta_1\|_2 = 0.5556$, which is identical to the optimal distance achieved in Case-1. 

The equality of the Frobenius and the spectral norms ($\|\Delta_1\|_F = \|\Delta_1\|_2 = 0.5556$) indicates that the optimal unstructured perturbation $\Delta_1$ is of rank-1 that produces the target eigenvalue $\lambda = 1.5181$ with  $AM(\lambda,Y_{\text{POF}}^{(1)} )=2$ and $GM(\lambda,Y_{\text{POF}}^{(1)})=1$. This is consistent with classical unstructured perturbation theory, which suggests that the nearest defective matrix in the spectral norm can often be achieved via a rank-1 update. The fact that the proposed method recovers this property while minimizing the Frobenius norm confirms that the algorithm has identified an optimal candidate in the unstructured search space.

\textbf{(b) Structured search space:} We now restrict the perturbation to strictly preserve the Hessenberg structure of $X$.  By modifying $D$ to project only onto the non-zero upper Hessenberg parameters, the proposed algorithm perfectly preserves the Hessenberg structure of the matrix, yielding the  optimal perturbed matrix as
$$
Y_{\text{POF}}^{(2)} = \begin{bmatrix}
 3.1442 & -2.0105 &  0.9866 &  3.9682 \\
-0.1373 & -3.0015 &  0.8719 &  1.0238 \\
 0 & -3.9425 &  2.1662 &  1.1542 \\
 0 &  0 &  5.0152 &  1.1367
\end{bmatrix}.
$$
The Frobenius and the spectral norms of $\Delta_2= Y_{\text{POF}}^{(2)} - X $ are given by
$$
\|\Delta_2\|_F = 0.9257 \quad \text{and}\quad \|\Delta_2\|_2 =   0.8845.
$$
As expected, restricting the feasible set increases the required perturbation norm ($\|\Delta\|_F$ increases from $0.5556$ to $0.9257$). Nevertheless, the optimal perturbed matrix $Y_{\text{POF}}^{(2)}$ preserves the Hessenberg structure while achieving the target eigenvalue $\lambda = 1.5181$ with  $AM(\lambda,Y_{\text{POF}}^{(2)})=2$ and  $GM(\lambda,Y_{\text{POF}}^{(2)})=1$.

\medskip
\noindent
\textit{Sub-case 2.2: Unknown eigenvalue (Problem~\eqref{eq:prob1}, $\mu$-optimized)}

Next, we free the eigenvalue parameter and repeat the experiments in Sub-case 2.1. This allows the algorithm to search the entire spectrum without prior knowledge of the eigenvalues.

\textbf{(a) Unstructured search space:} Upon solving, the algorithm converges to $\mu^\star = 1.5202$  with $AM(\mu^\star,Y_{\text{POF}}^{(3)} )=2$ and $GM(\mu^\star,Y_{\text{POF}}^{(3)} )=1$, where the optimal perturbed matrix $Y_{\text{POF}}^{(3)} = X + \Delta_3$ is given as
$$
Y_{\text{POF}}^{(3)}  = \begin{bmatrix}
  3.0444 & -2.0113 &  0.9980 &  3.9778 \\
 -0.7000 & -3.0763 &  0.9868 &  0.8498 \\
 -0.3761 & -3.9044 &  2.0165 &  1.1883 \\
 -0.0307 &  0.0078 &  5.0014 &  1.0154
\end{bmatrix}.
$$
The Frobenius and spectral norms of 
$\Delta_3$ are given by 
$$
\|\Delta_3\|_F = 0.5556, \quad \|\Delta_3\|_2 = 0.5556.
$$
This implies that the optimal unstructured perturbation $\Delta_3$ is of rank-1. The algorithm automatically finds the optimal perturbation and achieves the distance $0.5556$. 

\textbf{(b) Structured search space:} By restricting the perturbation to the subspace of upper Hessenberg matrices, the algorithm finds the nearest structured matrix 
$$
Y_{\text{POF}}^{(4)} = \begin{bmatrix}
    3.1158 & -2.0277 &  1.0006 &  3.9219 \\
   -0.5233 & -3.1141 &  1.0026 &  0.6785 \\
         0 & -3.8652 &  1.9969 &  1.3800 \\
         0 &       0 &  5.0000 &  1.0014
\end{bmatrix}.
$$
with an eigenvalue $\mu^\star = 0.9612$ such that $AM(\mu^\star ,Y_{\text{POF}}^{(4)})=2$ and $GM(\mu^\star ,Y_{\text{POF}}^{(4)})=1$. 

By allowing only structured perturbation naturally requires a greater perturbation norm than the unstructured case
$$
\|\Delta_4\|_F = 0.7256, \quad \|\Delta_4\|_2 = 0.6610.
$$


To verify the internal consistency of the proposed method, we
take the optimal eigenvalues discovered autonomously in Sub-case 2.2
and run the experiments of Sub-case 2.1. For the unstructured search space (Problem~\eqref{eq:prob2} with $\lambda = 1.5202$), our algorithm recovers the exact optimal updated matrix $Y_{\text{POF}}^{(3)}$ found in the free-eigenvalue case. Similarly, for the structured search space (Problem~\eqref{eq:prob2} with $\lambda = 0.9612$), algorithm recovers the exact upper Hessenberg matrix $Y_{\text{POF}}^{(4)}$. These results demonstrate that the formulation is consistent; providing the optimal eigenvalue as a priori knowledge yields the exact same solutions as when the algorithm searches the spectrum autonomously. These results are summarized in Table~\ref{tab:comprehensive_comparison}.

\begin{table}[htbp]
\centering
\begin{tabular}{llccc}
\hline
Method & Target eigenvalue & Structure preserved & $\|\Delta\|_2$ & $\|\Delta\|_F$  \\
\hline
SVD-based (\cite{mengi2011locating})  & Free ($\mu^\star=1.5181$) & No & 0.5556 & 0.7330   \\
POF & Fixed ($\lambda=1.5181$) & No & 0.5556 & 0.5556  \\
POF & Fixed ($\lambda=1.5181$) & Yes & 0.8845 & 0.9257   \\
POF & Free ($\mu^\star = 1.5202$) & No & 0.5556 & 0.5556  \\
POF & Free ($\mu^\star = 0.9612$) & Yes & 0.6610 & 0.7256  \\
POF & Fixed ($\lambda = 1.5202$) & No & 0.5556 & 0.5556  \\
POF & Fixed ($\lambda = 0.9612$) & Yes & 0.6610 & 0.7256  \\
\hline
\end{tabular}
\caption{Comparison of POF and the existing literature for the matrix $X$ in Example~\ref{ex:comparison_literature}.}
\label{tab:comprehensive_comparison}
\end{table}

Next we analyze the movement of the eigenvalues of $X$ for Sub-case 2.1 with the fixed target eigenvalue $\lambda = 1.5181$ under structured perturbation.
The left plot in Figure~\ref{fig:eigenvalue_trajectories_case2} shows the unperturbed eigenvalues of the matrix $X$. The right plot traces the movement of the eigenvalues (surrounded by circles) of $X$ with respect to perturbation $X+t  \Delta_2$ as $t$ moves from $0$ to $1$, $\Delta_2$ being the minimal upper Hessenberg perturbation such that $X+\Delta_2$ has the prescribed eigenvalue $\lambda=1.5181$ (surrounded by a diamond) with $\mathrm{AM}(\lambda;Y^*_\lambda) = 2$ and $\mathrm{GM}(\lambda;Y^*_\lambda) = 1$. The eigenvalue curves originated from two complex conjugate eigenvalues move toward the real axis and meet exactly at the target eigenvalue $\lambda=1.5181$. This confirms that the algorithm successfully forces distinct eigenvalues to coalesce at a single point, making the updated matrix defective with $AM=2$ and $GM=1$ at the target eigenvalue.

\begin{figure}[ht]
\centering
\begin{subfigure}{0.50\textwidth}
    \centering
    \includegraphics[width=\linewidth]{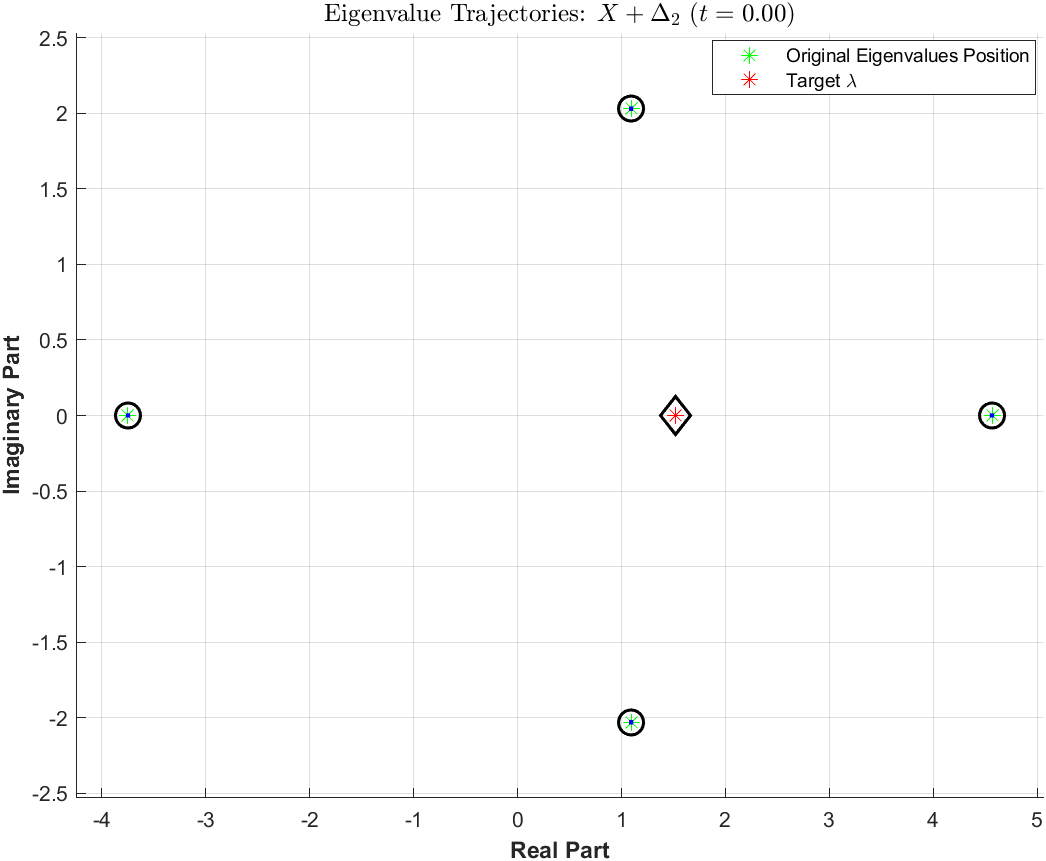}
    \caption{$t = 0.00$}
    \label{fig:T0_case2}
\end{subfigure}\hfill
\begin{subfigure}{0.50\textwidth}
    \centering
    \includegraphics[width=\linewidth]{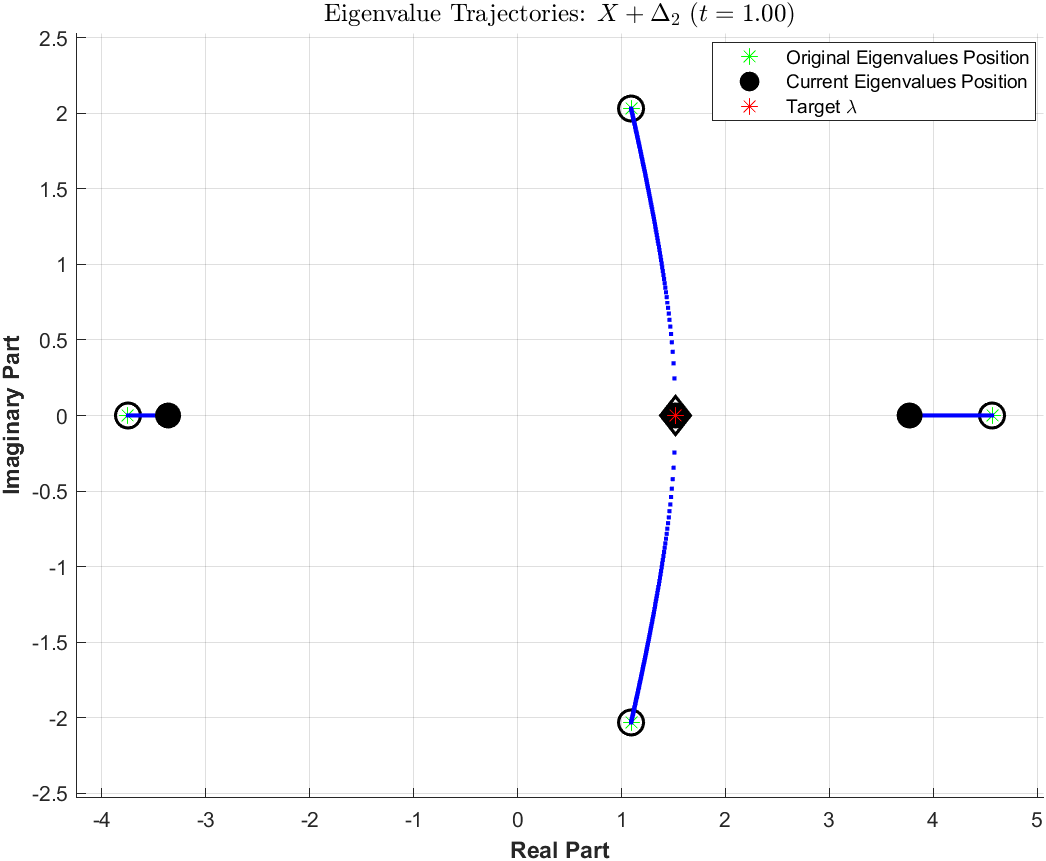}
    \caption{$t = 1.00$}
    \label{fig:T1_case2}
\end{subfigure}
\caption{Eigenvalue perturbation curves for the matrix $Y(t) = X + t \Delta_2$ in Example~\ref{ex:comparison_literature}. The left panel (a) shows the original eigenvalues before the perturbation ($t=0$). The right panel (b) shows the eigenvalue curves (dotted blue lines) and the final eigenvalues at $t=1$. Multiple eigenvalue curves meet exactly at the target eigenvalue $\lambda$ (surrounded by a diamond).}
\label{fig:eigenvalue_trajectories_case2}
\end{figure}

\end{example}

\section{Computational Scalability}

To assess the computational scalability of the proposed approach, numerical experiments were performed for matrix dimensions ranging from $n=5$ to $55$. For each matrix dimension $n$, the initial matrix was chosen as a random Toeplitz matrix generated from a randomly selected first row and first column. Also, for each problem size, the prescribed multiplicities were scaled proportionally according to
\[
m=\left\lfloor \frac{n}{2}\right\rfloor,
\qquad
p=\left\lfloor \frac{n}{4}\right\rfloor,
\]
thereby increasing the complexity of the optimization problem as the matrix dimension grows.

Figure~\ref{fig:scalability} presents the corresponding execution times. The results indicate that the proposed method remains computationally efficient for small and medium-sized problems ($n\leq 35$), with relatively modest increases in runtime. As the matrix size increases beyond this range, a more growth in computational cost is observed. This behavior is expected since larger values of $n$ lead to a substantial increase in the number of optimization variables and nonlinear constraints associated with the prescribed algebraic and geometric multiplicity conditions.

In particular, for $n\geq 50$, the runtime increases significantly, reflecting the higher computational burden of solving large-scale nonconvex optimization problems. Nevertheless, the algorithm successfully converged for all tested problem sizes.
\begin{figure}[htbp]
    \centering
    \includegraphics[width=0.8\textwidth]{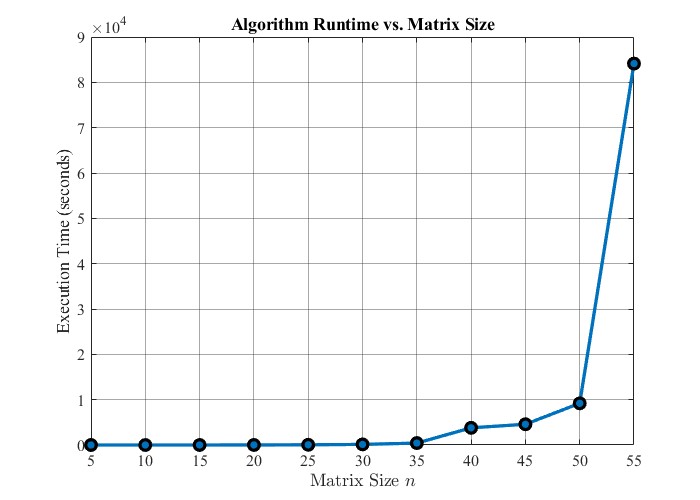}
    \caption{Execution time of Algorithm~\ref{alg:nearest_structured_AM_GM} versus matrix dimension $n$, with prescribed multiplicities scaled as $m=\lfloor n/2\rfloor$ and $p=\lfloor n/4\rfloor$.}
    \label{fig:scalability}
\end{figure}

\section{Conclusion}
\label{sec:conclusion}
In this work, we investigated how to compute the nearest structured matrix possessing an eigenvalue with specified lower bounds on both its algebraic and geometric multiplicities. We evaluated two distinct scenarios: one where the target eigenvalue is predetermined, and another where it acts as an unknown variable determined during the optimization process. By mapping Jordan chain relations into structural coordinates, we derived the exact necessary and sufficient conditions required for these structured perturbations to exist.
To minimize the Frobenius norm of the perturbation, we cast this task as a nested optimization problem constrained by the required multiplicity bounds. Finally, we implemented a nested optimization framework to reliably compute these optimal solutions. The extensive numerical experiments suggest that the proposed method accurately enforces multiplicity requirements while preserving the structure and maintaining minimal perturbations. Furthermore, numerical comparisons are presented to quantitatively validate the proposed framework against classical unstructured methods. Future work includes extensions to large-scale systems and polynomial eigenvalue problems.

\bigskip

\section*{Acknowledgments}
H.L. acknowledges the support from the University Grants Commission (UGC), Government of India, in the form of a Ph.D. fellowship (Award Reference: 221610039349). PS acknowledges the support of the SERB - CRG grant (CRG/2023/003221) by Government of India.

\section*{Declaration of competing interest} The authors declare no competing interests.

\section*{Data availability} No data was used for the research described in the article.

 \bibliographystyle{plain}
 \bibliography{ref}

@book{horn2012matrix,
  title={Matrix analysis},
  author={Horn, Roger A and Johnson, Charles R},
  year={2012},
  publisher={Cambridge University Press}
}

@book{trefethen2022numerical,
  title={Numerical linear algebra},
  author={Trefethen, Lloyd N and Bau, David},
  year={2022},
  publisher={SIAM}
}

@article{gray2006toeplitz,
  title={Toeplitz and circulant matrices: A review},
  author={Gray, Robert M and others},
  journal={Foundations and Trends in Communications and Information Theory},
  volume={2},
  number={3},
  pages={155--239},
  year={2006},
  publisher={Now Publishers, Inc.}
}

@Inbook{GilS2024,
author="Gillis, Nicolas
and Sharma, Punit",
editor="Guglielmi, Nicola
and Lubich, Christian",
title="Solving Matrix Nearness Problems via Hamiltonian Systems, Matrix Factorization, and Optimization",
bookTitle="Recent Stability Issues for Linear Dynamical Systems: Cetraro, Italy 2021",
year="2024",
publisher="Springer Nature Switzerland",
address="Cham",
pages="1--83",
}

@book{higham1988matrix,
  title={Matrix nearness problems and applications},
  author={Higham, Nicholas J},
  year={1988},
  publisher={University of Manchester. Department of Mathematics}
}

@article{higham2002computing,
  title={Computing the nearest correlation matrix—a problem from finance},
  author={Higham, Nicholas J},
  journal={IMA Journal of Numerical Analysis},
  volume={22},
  number={3},
  pages={329--343},
  year={2002},
  publisher={OUP}
}

@article{malyshev1999formula,
  title={A formula for the 2-norm distance from a matrix to the set of matrices with multiple eigenvalues},
  author={Malyshev, Alexander N},
  journal={Numerische Mathematik},
  volume={83},
  number={3},
  pages={443--454},
  year={1999},
  publisher={Springer}
}

@article{nazari2010computational,
  title={Computational aspect to the nearest matrix with two prescribed eigenvalues},
  author={Nazari, AM and Rajabi, D},
  journal={Linear Algebra and its Applications},
  volume={432},
  number={1},
  pages={1--4},
  year={2010},
  publisher={Elsevier}
}

@article{armentia2020nearest,
  title={Nearest matrix with a prescribed eigenvalue of bounded multiplicities},
  author={Armentia, Gorka and Gracia, Juan-Miguel and Velasco, Francisco-Enrique},
  journal={Linear Algebra and its Applications},
  volume={592},
  pages={188--209},
  year={2020},
  publisher={Elsevier}
}

@article{mengi2011locating,
  title={Locating a nearest matrix with an eigenvalue of prespecified algebraic multiplicity},
  author={Mengi, Emre},
  journal={Numerische Mathematik},
  volume={118},
  number={1},
  pages={109--135},
  year={2011},
  publisher={Springer}
}

@article{kokabifar2016nearest,
  title={Nearest matrix with prescribed eigenvalues and its applications},
  author={Kokabifar, Esmaeil and Loghmani, Ghasem Barid and Karbassi, Seyed-Mehdi},
  journal={Journal of Computational and Applied Mathematics},
  volume={298},
  pages={53--63},
  year={2016},
  publisher={Elsevier}
}

@book{borsdorf2012structured,
  title={Structured matrix nearness problems: Theory and algorithms},
  author={Borsdorf, Ruediger},
  year={2012},
  publisher={The University of Manchester (United Kingdom)}
}

@article{lalhriatpuia2025symmetric,
  title={Symmetric structured finite element model updating with prescribed partial eigenvalues while maintaining no spillover},
  author={Lalhriatpuia, H and Saha, Tanay},
  journal={Journal of Computational and Applied Mathematics},
  pages={116698},
  year={2025},
  publisher={Elsevier}
}

@book{higham1985nearness,
  title={Nearness Problems in Numerical Linear Algebra},
  author={Higham, Nicholas J},
  year={1985},
  publisher={The University of Manchester (United Kingdom)}
}

@article{higham1988computing,
  title={Computing a nearest symmetric positive semidefinite matrix},
  author={Higham, Nicholas J},
  journal={Linear Algebra and its Applications},
  volume={103},
  pages={103--118},
  year={1988},
  publisher={Elsevier}
}

@article{gracia2005nearest,
  title={Nearest matrix with two prescribed eigenvalues},
  author={Gracia, Juan-Miguel},
  journal={Linear Algebra and its Applications},
  volume={401},
  pages={277--294},
  year={2005},
  publisher={Elsevier}
}

@article{lippert2005fixing,
  title={Fixing two eigenvalues by a minimal perturbation},
  author={Lippert, Ross A},
  journal={Linear Algebra and its Applications},
  volume={406},
  pages={177--200},
  year={2005},
  publisher={Elsevier}
}

@article{kressner2014generalized,
  title={Generalized eigenvalue problems with specified eigenvalues},
  author={Kressner, Daniel and Mengi, Emre and Naki{\'c}, Ivica and Truhar, Ninoslav},
  journal={IMA Journal of Numerical Analysis},
  volume={34},
  number={2},
  pages={480--501},
  year={2014},
  publisher={Oxford University Press}
}

@article{mehrmann2001structure,
  title={Structure-preserving methods for computing eigenpairs of large sparse skew-{H}amiltonian/{H}amiltonian pencils},
  author={Mehrmann, Volker and Watkins, David},
  journal={SIAM Journal on Scientific Computing},
  volume={22},
  number={6},
  pages={1905--1925},
  year={2001},
  publisher={SIAM}
}

@phdthesis{nunez2019structured,
  title={Structured perturbation theory for eigenvalues of symplectic matrices},
  author={Nu{\~n}ez, Fredy Ernesto Sosa and Carre{\~n}o, Julio Moro},
  year={2019},
  school={Universidad Carlos III de Madrid}
}

@article{butta2015differential,
  title={Differential equations for real-structured defectivity measures},
  author={Butt{\`a}, Paolo and Guglielmi, Nicola and Manetta, Manuela and Noschese, Silvia},
  journal={SIAM Journal on Matrix Analysis and Applications},
  volume={36},
  number={2},
  pages={523--548},
  year={2015},
  publisher={SIAM}
}

@book{nocedal2006numerical,
  title={Numerical optimization},
  author={Nocedal, Jorge and Wright, Stephen J},
  year={2006},
  publisher={Springer}
}

@book{dempe2002foundations,
  title={Foundations of bilevel programming},
  author={Dempe, Stephan},
  year={2002},
  publisher={Springer}
}

@article{orucc2016number,
  title={On number of partitions of an integer into a fixed number of positive integers},
  author={Oru{\c{c}}, A Yavuz},
  journal={Journal of Number Theory},
  volume={159},
  pages={355--369},
  year={2016},
  publisher={Elsevier}
}

@article{Saha24082026,
author = {Tanay Saha},
title = {Nearest structured polynomial matrix having an eigenvalue with prescribed geometric multiplicity},
journal = {Linear and Multilinear Algebra},
volume = {0},
number = {0},
pages = {1--28},
year = {2026},
publisher = {Taylor \& Francis},
doi = {10.1080/03081087.2026.2722664},


URL = { 
    
        https://doi.org/10.1080/03081087.2026.2722664
    
    

},
eprint = { 
    
        https://doi.org/10.1080/03081087.2026.2722664
    
    

}

}
\end{document}